\documentclass[final]{siamltex}
\usepackage{}
\usepackage{amsfonts}
\usepackage{bbm}
\usepackage{amsmath}
\usepackage{mathrsfs}
\usepackage{epsfig}
\usepackage{graphicx}
\usepackage{subfigure}
\usepackage{epstopdf}
\usepackage{cases}

\graphicspath{{FIG/}}

\usepackage{amsmath,amssymb,enumerate,color}

\usepackage{comment,enumerate,multicol,xspace}

\title{Stability and error estimates for the variable step-size BDF2 method for linear and semilinear parabolic equations\thanks{This work was supported by a grant from the National Natural Science
Foundation of China (Grant No. 11771060). }}

\author{Wansheng Wang\thanks{Department of Mathematics, Shanghai Normal University, 200234, Shanghai, China ({\tt w.s.wang@163.com}).}
                \and Mengli Mao\thanks{Department of Mathematics, Shanghai Normal University, 200234, Shanghai, China.}
                \and Zheng Wang\thanks{Department of Mathematics, Shanghai Normal University, 200234, Shanghai, China.}
                }

\begin{document}

\maketitle

\begin{abstract}
In this paper stability and error estimates for time discretizations of linear and semilinear parabolic equations by the two-step backward differentiation formula (BDF2) method with variable step-sizes are derived. An affirmative answer is provided to the question: whether the upper bound of step-size ratios for the $l^\infty(0,T;H)$-stability of the BDF2 method for linear and semilinear parabolic equations is identical with the upper bound for the zero-stability. The $l^\infty(0,T;V)$-stability of the variable step-size BDF2 method is also established under more relaxed condition on the ratios of consecutive step-sizes. Based on these stability results, error estimates in several different norms are derived. To utilize the BDF method the trapezoidal method and the backward Euler scheme are employed to compute the starting value. For the latter choice, order reduction phenomenon of the constant step-size BDF2 method is observed theoretically and numerically in several norms. Numerical results also illustrate the effectiveness of the proposed method for linear and semilinear parabolic equations.
\end{abstract}



\pagestyle{myheadings} \thispagestyle{plain} \markboth{WANSHENG WANG, MENGLI MAO, AND ZHENG WANG}{VARIABLE STEP-SIZE BDF2 FOR PDE}

\section{Introduction} In this paper we shall study stability and error estimates for time discretizations by the two-step backward differentiation formula (BDF2) with variable step-sizes for linear parabolic partial differential equations (PDEs)
\begin{equation}\label{eq1.1}
\left\{
\begin{array}{lll}
u^{\prime}(t)+Au(t)+Bu(t)=f(t),\qquad t\in J:=(0,T],\\
u(0)=u^{0},\\
\end{array}
\right.
\end{equation}
and its semilinear extension, where $A$: $D(A)\rightarrow H$ is a positive definite, self-adjoint, linear operator on a Hilbert space $(H,(\cdot,\cdot))$ with domain $D(A)$ dense in $H$, the linear operator $B: D(A)\to H$ satisfies some structural assumptions; here the forcing term $f:[0,T]\to H$, and initial value $u^{0}\in H$.

Let the time interval $[0,T]$ for given $N\in \mathbb N$, $N\ge 2$, be partitioned via $0=t^0<t^1<\cdots<t^N=T$. Let $k_{n}=t^{n}-t^{n-1},~n=1,2,\dotsc, N$, be the time step-sizes which in general will be variable, and $J_{n}:=(t^{n-1},t^{n}],~n=1,2,\dotsc, N$. We set
\begin{eqnarray*}
r_n=\frac{k_n}{k_{n-1}},~~ n=2,3,\dotsc,N;\quad k_{\max}=\max\limits_{n=2,3,\dotsc,N}k_n;\quad r_{\max}=\max\limits_{n=2,\dotsc,N}r_n.
\end{eqnarray*}
Assuming we are given starting approximations $U^0$ and $U^1$, which is computed by the trapezoidal method or the backward Euler
scheme, we discretize (1.1) in time by the variable step-size BDF2, i.e., we define nodal approximations $U^n\in D(A)$ to the values $u^n:=u(t^n)$ of the exact solution $u$ to (1.1) as follows:
\begin{eqnarray}\label{eq2.5}
\bar\partial^2_BU^{n}+AU^{n}+BU^{n}=f^{n},\qquad n=2,3,\dotsc,N,
\end{eqnarray}
where $f^n:=f(t^n)$ and
\begin{eqnarray*}
\bar\partial^2_BU^{n}&:=&(1+s_n)\bar\partial^1_BU^n-s_n\bar\partial^1_BU^{n-1}\\
&=&\frac{1}{k_{n}}\left((1+s_n)U^{n}-(1+r_{n})U^{n-1}+r_{n}s_nU^{n-2}\right).
\end{eqnarray*}
Here $\bar\partial^1_BU^n$ and $s_n$ are defined by
$$\bar\partial^1_BU^n=\bar\partial U^n:=\frac{U^n-U^{n-1}}{k_n}, \quad {\hbox{and}}~~s_n:=\frac{r_n}{1+r_n}=\frac{k_n}{k_n+k_{n-1}},$$
respectively. For an equidistant partition with $k_n=k$, we have $s_n=\frac{1}{2}$ and the well-known formula
$$\bar\partial^2_BU^n=\frac{1}{k}\left(\frac{3}{2}U^n-2U^{n-1}+\frac{1}{2}U^{n-2}\right).$$

The BDF2 method is one of the most popular time-stepping methods and many studies have been conducted on the stability and error estimates for it. Because of its good stability property (the scheme is $G$-stable), the BDF2 method with constant step-size has been dealt with for various equations as, e.g., linear parabolic equations \cite{Thomee97,Auzinger10}, integro-differential equations \cite{McLean93}, jump-diffusion model in finance \cite{Almendral05}, the Navier-Stokes equations \cite{Girault81,Hill00,Emmrich04}. When the solutions of time dependent differential equations have different time scales, i.e., solutions rapidly varying in some regions of time while slowly changing in other regions, variable step-sizes are often essential to obtain computationally efficient, accurate results. Owing to these prominent advantages, the variable step-size BDF2 method has been successfully applied to partial integro-differential equations \cite{Wang19} and Cahn-Hilliard equation \cite{Chen19} recently. An important result that the variable step-size BDF2 method is zero-stable if the step-size ratios $r_n$ are less than $R_0=\sqrt{2}+1\approx 2.414$  has been independently proved by several authors \cite{Zlatev81,Grigorieff83,Crouzeix84}. And the value of $R_0$ cannot be improved when dealing with arbitrary variable step-sizes (see, for example, \cite{Grigorieff83,Calvo90,Calvo93}).

For the variable step-size BDF2 method applied to linear parabolic equations, Le Roux \cite{LeRoux82} derived stability and error bounds in the $l^\infty(J;H)$ norm by using spectral techniques under the step-size conditions
 \begin{eqnarray}\label{eq1.3}
 c k_{\max}\le k_j\le k_{\max},\qquad |r_n-1|\le C\frac{k_{n-1}}{1+|\log k_{n-1}|},
  \end{eqnarray}
with constants $c\in (0,1]$ and $C>0$. Palencia and Garc\'ia-Archilla \cite{Palencia93} studied linear parabolic equations in a Banach space setting and obtained that the ratios $r_n$ should be close to $1$ such as, e.g., in (\ref{eq1.3}) for the stability factor to be moderate. Grigorieff \cite{Grigorieff91,Grigorieff95} showed stability and optimal error estimates with smooth or non-smooth data under the assumption that the step-size ratios are less than $(\sqrt{3}+1)/2\approx 1.366$ in a Banach space setting. Becker improved the bound up to $(\sqrt{13}+2)/3\approx 1.868$ in Hilbert space \cite{Becker98} (see, also, \cite{Thomee97}) by testing (\ref{eq2.5}) with $U^n_\delta:=U^n+\delta k_n\bar\partial^1_BU^n$ for a specified constant $\delta$. Based on the same technique, i.e., testing (\ref{eq2.5}) with $U^n_\delta:=U^n+\delta k_n\bar\partial U^n$, Emmrich \cite{Emmrich05} extended the results to semilinear parabolic problems and further improved the bound to $1.910$ using a more general identity for $2(AU^n,U^n_\delta)$. Emmrich \cite{Emmrich09} also studied the stability and convergence of the variable step-size BDF2 method for nonlinear evolution problems governed by a monotone potential operator.

It is natural to ask what the upper bound of step-size ratios is and whether it is identical with the upper bound $R_0$ for the zero-stability. In this paper we will address this question and give an affirmative answer. We explore a new technique, which is very different from the one used by Becker \cite{Becker98}, Thom\'ee \cite{Thomee97} and Emmrich \cite{Emmrich05,Emmrich09}. We first test (\ref{eq2.5}) with $\bar\partial U^n$ to obtain $l^\infty(J;V)$-stability and $l^2(J;H,H)$-stability (their definitions will be introduced in Section 2). Then after we test (\ref{eq2.5}) with $U^n$, using $l^2(J;H,H)$-stability estimate, we obtain the usual stability result in the $l^\infty(J;H)$ norm under the sharp zero-stability condition on the ratios of consecutive step-sizes. Following the approach of Chen et. al. \cite{Chen19}, the $l^\infty(J;V)$ and $l^2(J;V)$-stabilities of the variable step-size BDF2 method are also established under a more relaxed assumption that the step-size ratios are less than $R_1=(3+\sqrt{17})/2\approx 3.561$.

It is well known that the method (\ref{eq2.5}) yields second order approximations $U^n$ to $u^n$ (in the $H$ norm) when the backward Euler method is used to compute the starting value $U^1$, since it is applied only once. This choice for $U^1$ is quite popular in the multistep methods for computations of parabolic equations. However, the error bounds derived in this paper based on the obtained stability results suggest that this is not the best choice for the constant step-size BDF2 method, since it will cause the reduction of the convergence order in the $l^\infty(J;V)$ and $l^2(J;H,H)$ norms. This will be discussed in detail in Section 4.

The rest of this paper is organized as follows. We start in Section 2 by introducing the necessary notation and recalling a lemma which will be used in the following analysis. The stability of the method in several norms under the condition that the step-size ratios are less than $R_0$ (or $R_1$) is proved in Section 3. Error estimates in different norms are derived in Section 4. Since our error estimates will depend on the first step error $U^1-u^1$, the error $U^1-u^1$ produced by the trapezoidal method or the backward Euler
scheme will be analyzed in this section too. Section 5 will extend the analysis to the semilinear case
\begin{equation}\label{eq1.1a}
\left\{
\begin{array}{lll}
u^{\prime}(t)+Au(t)=f(t,u),\qquad t\in J:=(0,T],\\
u(0)=u^{0},\\
\end{array}
\right.
\end{equation}
with some assumptions on the nonlinear operator $f(t,\cdot)$. Section 6 is devoted to numerical experiments, which confirm our theoretical results and illustrate the effectiveness of the proposed method for linear and semilinear parabolic equations. Section 7 contains a few concluding remarks.

\section{Variable step-size BDF2 method for linear parabolic equations}
\label{2}

Now we consider the variable two-step BDF method for solving (1.1). To do this, we first make some assumptions and introduce the necessary notation.

\subsection{Linear parabolic equations} Let $ V:=D(A^{\frac{1}{2}})$ and denote the norms in $H$ and $V$ by $|\cdot|$ and $ \|\cdot\|$, $ \|v\|=|A^{\frac{1}{2}}v|=(Av,v)^{\frac{1}{2}}$, respectively. Let $ V^{*}$ be the dual of $V$ ($V\subset H\subset V^{*}$), and denote by $\|\cdot\|_*$ the dual norm on $V^*$, $\|v\|_{*}=|A^{-\frac{1}{2}}v|=(v,A^{-1}v)^{\frac{1}{2}}$. We denote by $(\cdot,\cdot)$ the duality pairing between $V^*$ and $V$. We define a bilinear form $a(\cdot,\cdot):V\times V\to \mathbb R$ via $(Au,v)=a(u,v)$. For the linear operator $B$, we assume that \begin{eqnarray}\label{eq2.3}
 |Bu|\le \gamma(t)\|u\|,\qquad \forall u\in V,\quad t\in J,\end{eqnarray}
with a smooth nonnegative function $\gamma: J\to \mathbb R$. Let $\gamma=\max_{t\in J}\gamma(t)$. We may write the parabolic problem in variational form as
 \begin{eqnarray}\label{eq2.4a}
 (u_t,v)+a(u,v)+(Bu,v)=(f,v),\quad \forall v\in V,~~t\in J;\qquad u(0)=u^0.
 \end{eqnarray}
Emmrich in \cite{Emmrich05} has shown that for given $u^0\in H$ and $f\in L^2(J;V^*)$, problem (\ref{eq2.4a}) admits a unique solution $u\in L^2(J;V)\cap C([0,T];H)$ with $u^\prime\in L^2(J;V^*)$.

{\em Standard example.} Let $A$ and $B$ be defined by
\begin{eqnarray*}
Au:=-\sum_{i,j=1}^d\frac{\partial}{\partial x_i}\left(\alpha_{ij}\frac{\partial u}{\partial x_j}\right),\qquad Bu:=\sum_{j=1}^d\alpha_j\frac{\partial u}{\partial x_j}+\alpha_0u,
\end{eqnarray*}
respectively, where $\alpha_{ij},~\alpha_j,~\alpha_0$ are sufficiently smooth functions in $x\in \Omega$ with $\Omega$ being a bounded domain in $\mathbb R^d$ with sufficiently smooth boundary $\partial \Omega$. Let $V=H^1_0(\Omega)$ and $H=L^2(\Omega)$ be the usual Sobolev and Lebesgue space, respectively. Assume that $(\alpha_{ij})$ is symmetric and uniformly positive definite. Then the operators $A$ is a positive definite, self-adjoint, linear operator, and $B$ satisfies the condition (\ref{eq2.3}).

\subsection{Variable step-size BDF2 method for linear parabolic equations} For the method (\ref{eq2.5}) we need the starting values $U^0$ and $U^1$. We set $U^0:=u^0$ and perform an initial trapezoidal approximation to get $U^1$
\begin{eqnarray}\label{eq2.4}\bar{\partial}U^{1}+AU^{\frac{1}{2}}+BU^{\frac{1}{2}}=f^{\frac{1}{2}}
\end{eqnarray}
with $v^{n-\frac{1}{2}}=\frac{v^n+v^{n-1}}{2}$. Note that with $r_n:=0$, the two-step BDF formally degenerates to a backward Euler step. It is also easy to see that
\begin{equation}\label{eq2.6}
\begin{split}
\bar\partial^2_BU^{n}&=k_{n}\bar{\partial}^{2}U^{n}+\frac{1}{k_{n-1}+k_{n}}(U^{n}-U^{n-2})\\
&=k_{n}\bar{\partial}^{2}U^{n}+
s_n\bar{\partial}U^{n}+\frac{k_{n-1}}{k_{n-1}+k_{n}}\bar{\partial}U^{n-1}\\
&=k_{n}s_n\bar{\partial}^{2}U^{n}+\bar\partial U^n,
\end{split}
\end{equation}
where
\begin{eqnarray*}\bar{\partial}^{2}v^{n}:=\bar\partial\bar\partial v^n=\frac{1}{k_n}\left[\bar\partial v^n-\bar\partial v^{n-1}\right].
\end{eqnarray*}

With respect to the solvability of the time discrete problem, Emmrich has shown in \cite{Emmrich05} that for given $U^0,U^1\in H$ and $f^n\in l^2(J;V^*)$, the problem
\begin{eqnarray}\label{eq2.8}
 (\bar\partial^2_BU^n,v)+a(U^n,v)+(BU^n,v)=(f^n,v),\quad \forall v\in V
 \end{eqnarray}
 admits a unique solution. For the obtained solution sequence $\{U^j\}_{j=n_1}^{n_2}$, we define the $l^\infty(t^{n_1},t^{n_2};H)$, $l^\infty(t^{n_1},t^{n_2};V)$ and $l^2(t^{n_1},t^{n_2};V)$ norms as
 \begin{eqnarray}
 |U|_{l^\infty(t^{n_1},t^{n_2};H)}:=\max\limits_{n_1\le j\le n_2}|U^j|,\qquad  \|U\|_{l^\infty(t^{n_1},t^{n_2};V)}:=\max\limits_{n_1\le j\le n_2}\|U^j\|,
 \end{eqnarray}
 and
 \begin{eqnarray}
 \|U\|_{l^2(t^{n_1},t^{n_2};V)}:=\left(\sum\limits_{j=n_1}^{n_2}k_j\|U^j\|^2\right)^{1/2},
 \end{eqnarray}
 respectively. It is well known that they are the discrete counterparts of the $L^\infty(t^{n_1},t^{n_2};H)$, $L^\infty(t^{n_1},t^{n_2};V)$ and $L^2(t^{n_1},t^{n_2};V)$ norms, respectively.

{\em Remark.} [The choice for $U^1$]. The starting value $U^1$ can be also obtained by the backward Euler method
\begin{equation}\label{eq2.9}
\bar\partial U^1+AU^1+BU^1=f^1
\end{equation}
 with $U^0=u^0$. It is well known that the constant step-size BDF2 method (\ref{eq2.5}) with this initial approximation $U^1$ also yields second order approximations $U^n$ to $u^n$ in $H$ norm; see Corollary 4.6 in Section 4, or, \cite{Becker98,Thomee97}. However, we find that order reduction will be caused in the $l^\infty(J;V)$ and $l^2(J;H,H)$ norms when the starting value $U^1$ is obtained by the backward Euler method (\ref{eq2.9}). Here the $l^2(t^{n_1},t^{n_2};H,H)$ norm of a solution sequence $\{U^j\}_{j=n_1}^{n_2}$ with $U^j\in H$ is defined by
 \begin{eqnarray}
 \|U\|_{l^2(t^{n_1},t^{n_2};H,H)}:=\left(\sum\limits_{j=n_1}^{n_2-1}k_js_j\left|\bar\partial U^{j}\right|^2\right)^{1/2},
 \end{eqnarray}
 which is the discrete counterpart of $\int_{t^{n_1}}^{t^{n_2}}|u^\prime(t)|^2dt$ with nonuniform grid weights $s_j$.

 Because of the different choices for $U^1$ and $k_1$, we pay special attention to $k_1$ and set $k_{\max}=\max\limits_{n=2,3,\dotsc,N}k_n$.

In subsequent sections, by convention, we set $\sum_{j=m}^nx_j=0$ and $\prod_{j=m}^nx_j=1$ if $m>n$. We will use the identity
\begin{eqnarray}\label{eq2.11}
2(\kappa a-\varepsilon b)a=(2\kappa-\varepsilon)a^2-\varepsilon b^2+\varepsilon (a-b)^2,\quad \kappa,\varepsilon\ge 0.
\end{eqnarray}
We also need the following discrete Gronwall lemma proved in \cite{Emmrich05}.

\begin{lemma}[Discrete Gronwall lemma \cite{Emmrich05}]\label{lem2.3} Let $0\le \lambda<1$, and $a_n,~b_n,~g_n,~\lambda_n\ge 0$ with $\{g_n\}$ being monotonically increasing. Then \begin{eqnarray}
a_n+b_n\le \sum\limits_{j=\varpi}^{n-1}\lambda_ja_j+\lambda a_n+ g_n,\qquad n=\varpi,\varpi+1,\dotsc
\end{eqnarray}
implies for $n=\varpi,\varpi+1,\dotsc$
\begin{eqnarray*}
a_n+b_n\le \frac{g_n}{1-\lambda}\prod\limits_{j=\varpi}^{n-1}\left(1+\frac{\lambda_j}{1-\lambda}\right)\le \frac{g_n}{1-\lambda}\exp\left(\frac{1}{1-\lambda}\sum\limits_{j=\varpi}^{n-1}\lambda_j\right).
\end{eqnarray*}
\end{lemma}

\section{Stability analysis}
In this section we shall show stability of the variable step-size BDF2 method with $r_{\max}<R_0$ applied to linear parabolic equations (1.1).
As mentioned in Introduction, for the variable step-size BDF2 method applied to parabolic equations, the best known result is that it is stable in the $l^\infty(J;H)$ norm when the step-size ratios are less than $1.910$. To improve the bound to $R_0=\sqrt{2}+1$, the upper bound for the zero-stability, we first need the following stability results in the $l^2(J;H,H)$ norm. Additionally, since the $V$ norm is an energy norm, from a physical point of view, $l^\infty(J;V)$ stability is of utmost important.

\begin{theorem}[$l^\infty(J;V)$ and $l^2(J;H,H)$ stability under $R<R_0$]\label{lem3.1} Let $r_{\max}\le R$ with $1<R<R_0=\sqrt{2}+1\approx 2.414$. If there exists a constant $c_1\in (0,1)$ such that $k_{\max}$ satisfies
\begin{eqnarray}\label{eq7.2}
(4+2\sqrt{2})\gamma^2 k_{\max}\le c_1<1,
\end{eqnarray}
then we have, for any $n\ge 2$,
\begin{equation}\label{eq7.3}
\begin{split}
&k_n\left|\bar\partial U^n\right|^2
+\left\|U\right\|^2_{l^2(t^2,t^n;H,H)}+\|U\|^2_{l^\infty(t^2,t^n;V)}\\
\le&C_1\left(\sum\limits_{j=2}^{n}k_j\left|f^{j}\right|^2
+k_2s_2\left|\bar\partial U^{1}\right|^2+\|U^1\|^2\right),
\end{split}
\end{equation}
where
$$C_1=\frac{4+2\sqrt{2}}{1-c_1}\exp\left(\frac{4+2\sqrt{2}}{1-c_1}\sum\limits_{j=2}^n\gamma^2(t^j)k_j\right)\le \frac{4+2\sqrt{2}}{1-c_1}\exp\left(\frac{4+2\sqrt{2}}{1-c_1}\gamma^2t^n\right).$$
\end{theorem}

\begin{proof}
We first have from (\ref{eq2.5}) and (\ref{eq2.6})
\begin{eqnarray}\label{eq3.3a}
k_ns_n\bar\partial^2U^n+\bar\partial U^n+AU^n+BU^n=f^n.
\end{eqnarray}
Taking in (\ref{eq3.3a}) the inner product with $2k_n\bar\partial U^n$, using the relation (\ref{eq2.11}), we obtain, for any $\epsilon\in (0,1)$,
\begin{equation}\label{eq3.4a}
\begin{split}
&k_ns_n\left(\left|\bar\partial U^n\right|^2-\left|\bar\partial U^{n-1}\right|^2+\left|\bar\partial U^n-\bar\partial U^{n-1}\right|^2\right)+2k_n\left|\bar\partial U^{n}\right|^2\\
&+\|U^{n}\|^2-\|U^{n-1}\|^2+\|U^{n}- U^{n-1}\|^2\\
\le&2\gamma(t^n) k_n\|U^n\|\left|\bar\partial U^{n}\right|+2k_n\left|f^{n}\right|\left|\bar\partial U^n\right|\\
\le&\frac{1}{\epsilon}\gamma^2(t^n)k_n\|U^n\|^2+\epsilon k_n\left|\bar\partial U^{n}\right|^2+\frac{1}{\epsilon}k_n\left|f^{n}\right|^2+\epsilon k_n\left|\bar\partial U^n\right|^2.
\end{split}
\end{equation}
Summing up gives for $n=2,3,\dotsc,N$,
\begin{eqnarray}\label{eq7.8a}
&&[s_n+2(1-\epsilon)]k_n\left|\bar\partial U^n\right|^2
+\sum\limits_{j=2}^{n-1}\left(k_js_j+2(1-\epsilon)k_j-k_{j+1}s_{j+1}\right)\left|\bar\partial U^{j}\right|^2+\|U^{n}\|^2\nonumber\\
&\le&\frac{1}{\epsilon}\sum\limits_{j=2}^{n}\gamma^2(t^j)k_j\|U^j\|^2+\frac{1}{\epsilon}\sum\limits_{j=2}^{n}k_j\left|f^{j}\right|^2+k_2s_2\left|\bar\partial U^{1}\right|^2+\|U^1\|^2.
\end{eqnarray}
Now take $\epsilon=\frac{1}{2}-\frac{\sqrt{2}}{4}$ such that
$$\sup\limits_{0\le x\le R_0}\frac{x^2}{1+x}\le 2(1-\epsilon).$$
Then we get
\begin{equation}\label{eq7.8}
\begin{split}
&(s_n+1)k_n\left|\bar\partial U^n\right|^2
+\sum\limits_{j=2}^{n-1}k_js_j\left|\bar\partial U^{j}\right|^2+\|U^{n}\|^2\\
\le&(4+2\sqrt{2})\sum\limits_{j=2}^{n}\gamma^2(t^j)k_j\|U^j\|^2
+(4+2\sqrt{2})\sum\limits_{j=2}^{n}k_j\left|f^{j}\right|^2
+k_2s_2\left|\bar\partial U^{1}\right|^2+\|U^1\|^2.
\end{split}
\end{equation}
An application of Lemma \ref{lem2.3} leads to the desired result.
\end{proof}

Recently, the $l^\infty(J;V)$ stability of a linearly implicit stabilization BDF2 method with variable step-sizes has been established under the condition $R<R_1=(3+\sqrt{17})/2\approx 3.561$ for the Cahn-Hilliard equation in \cite{Chen19}. Following their approach, we can also improve the bound to $R_1=(3+\sqrt{17})/2\approx 3.561$ for the $l^\infty(J;V)$ stability of the variable step-size BDF2 method for the problem (1.1).

\begin{theorem}[$l^\infty(J;V)$ and $l^2(J;V)$ stability under $R<R_1$] Let $r_{\max}\le R$ with $1<R<R_1=\frac{3+\sqrt{17}}{2}$, and let $c_R>\max\{\frac{2+2R}{2+R},\frac{2+2R}{2+3R-R^2}\}$. If there exists a constant $c_1\in (0,1)$ such that $k_{\max}$ satisfies
\begin{eqnarray}\label{eq3.7a}
c_R\gamma^2 k_{\max}\le c_1<1,
\end{eqnarray}
then we have, for any $n\ge 2$,
\begin{equation}\label{eq3.8aa}
\begin{split}
\frac{R}{1+R}k_n\left|\bar\partial U^n\right|^2+\|U^{n}\|^2\le C_2\left(\frac{R}{1+R}k_{1}\left|\bar\partial U^{1}\right|^2+ \|U^{1}\|^2+c_R\sum\limits_{j=2}^nk_j\left|f^{j}\right|^2\right),
\end{split}
\end{equation}
and
\begin{equation}\label{eq3.3aa}
\begin{split}
\sum\limits_{j=2}^nk_j\|U^{j}\|^2\le C_2t^n\left(k_{1}\left|\bar\partial U^{1}\right|^2+ \|U^{1}\|^2+c_R\sum\limits_{j=2}^nk_j\left|f^{j}\right|^2\right),
\end{split}
\end{equation}
where
$$C_2=\frac{c_R}{1-c_1}\exp\left(\frac{c_R}{1-c_1}\sum\limits_{j=2}^n\gamma^2(t^j)k_j\right)\le \frac{c_R}{1-c_1}\exp\left(\frac{c_R}{1-c_1}\gamma^2t^n\right).$$
\end{theorem}

\begin{proof}
Taking in (\ref{eq3.3a}) the inner product with $2k_n\bar\partial U^n$, we obtain, for any $\epsilon\in (0,1)$,
\begin{equation}
\begin{split}
&\frac{k_n(2+4r_n-r_n^2)}{1+r_n}|\bar\partial U^n|^2-\frac{k_n}{1+r_n}|\bar\partial U^{n-1}|^2+r_ns_n\frac{|U^n-2U^{n-1}+U^{n-2}|^2}{k_n}\\
&+\|U^{n}\|^2-\|U^{n-1}\|^2+\|U^{n}- U^{n-1}\|^2\\
\le&2\gamma(t^n) k_n\|U^n\|\left|\bar\partial U^{n}\right|+2k_n\left|f^{n}\right|\left|\bar\partial U^n\right|\\
\le&\frac{1}{\epsilon}\gamma^2(t^n)k_n\|U^n\|^2+\epsilon k_n\left|\bar\partial U^{n}\right|^2+\frac{1}{\epsilon}k_n\left|f^{n}\right|^2+\epsilon k_n\left|\bar\partial U^n\right|^2.
\end{split}
\end{equation}
Ignoring some of the positive terms on the left-hand side, we have
\begin{equation}
\begin{split}
&g(z,\epsilon)k_n\left|\bar\partial U^n\right|^2+\|U^{n}\|^2\\
\le&s_nk_{n-1}\left|\bar\partial U^{n-1}\right|^2+ \|U^{n-1}\|^2+\frac{1}{\epsilon}\gamma^2(t^n)k_n\|U^n\|^2+\frac{1}{\epsilon}k_n\left|f^{n}\right|^2,
\end{split}
\end{equation}
where $g(r_n,\epsilon)=\frac{2+4r_n-r_n^2}{1+r_n}-2\epsilon$. In the case $r_n\le 2$, noting that $g(r_n,\epsilon)\ge 2-2\epsilon$, we can take $0<\epsilon<1-\frac{R}{2(1+R)}$ such that
\begin{eqnarray}\label{eq3.11}
g(r_n,\epsilon)>\frac{R}{1+R}.
 \end{eqnarray}
 In the case $2<r_n<\frac{3+\sqrt{17}}{2}$, since $\frac{2+4r_n-r_n^2}{1+r_n}$ is a decreasing function, we take $0<\epsilon<\frac{2+4R-R^2}{2(1+R)}-\frac{R}{2(1+R)}=\frac{2+3R-R^2}{2(1+R)}$ such that (\ref{eq3.11}) holds. Thus in both cases, we have
\begin{equation}\label{eq3.12}
\begin{split}
&\frac{R}{1+R}k_n\left|\bar\partial U^n\right|^2+\|U^{n}\|^2\\
\le&\frac{R}{1+R}k_{n-1}\left|\bar\partial U^{n-1}\right|^2+ \|U^{n-1}\|^2+\frac{1}{\epsilon}\gamma^2(t^n)k_n\|U^n\|^2+\frac{1}{\epsilon}k_n\left|f^{n}\right|^2\\
\le&\frac{R}{1+R}k_{1}\left|\bar\partial U^{1}\right|^2+ \|U^{1}\|^2+\frac{1}{\epsilon}\sum\limits_{j=2}^n\gamma^2(t^j)k_j\|U^j\|^2+\frac{1}{\epsilon}\sum\limits_{j=2}^nk_j\left|f^{j}\right|^2,
\end{split}
\end{equation}
Apply {\sc Lemma \ref{lem2.3}} to (\ref{eq3.12}) to obtain the desired inequality \eqref{eq3.8aa}. The estimate \eqref{eq3.3aa} is a direct result of \eqref{eq3.8aa}.  This completes the proof.
\end{proof}

Now we use the $l^2(J;H,H)$ stability estimate (\ref{eq7.3}) in {\sc Theorem 3.1} to show the stability of the variable step-size BDF2 method in the $l^\infty(J;H)$ and $l^2(J;V)$ norms.

\begin{theorem}[$l^\infty(J;H)$ and $l^2(J;V)$ stability under $R<R_0$]\label{th1}
Let $r_{\max}\le R$ with $1<R<R_0=\sqrt{2}+1\approx 2.414$. If there exist constants $c_1$ and $c_2$ such that $k_{\max}$ satisfies \eqref{eq7.2} and
\begin{equation}\label{eq3.7} 2c_3\gamma^2k_{\max}\le c_2<1,
\end{equation}
where $c_3=\frac{(1+R)^2}{1+2R-R^2}$, then the following estimate holds for $n=2,3,\dotsc,N$:
\begin{equation}\label{eq3.1}
\begin{split}
|U|^{2}_{l^\infty(t^2,t^n;H)}+\|U\|^2_{l^2(t^2,t^n;V)}\le C f_U,
\end{split}
\end{equation}
with
$$f_U=|U^{1}|^2+k_{\max} \|U^1\|^2+(k_1^2+k_{\max}k_2)|\bar\partial U^{1}|^2+\sum\limits_{j=2}^nk_j\left(\|f^j\|_*^2+k_{\max}\left|f^{j}\right|^2\right).$$
Here, $C$ depends on $\gamma$, $c_i$, $i=1,2,3$, and $R$, $T$, $\Phi_n$ with $\Phi_n$ being defined by
$$\Phi_n:=\sum\limits_{j=2}^{n-2}[r_j-r_{j+2}]_+,\qquad [x]_+:=\frac{|x|+x}{2}.$$
\end{theorem}

\begin{proof} Taking in (\ref{eq2.5}) the inner product with $\frac{2k_n}{1+r_n}U^n$ yields
\begin{equation}\label{eq3.8a}
\begin{split}
&\frac{2k_n}{1+r_n}\left(\bar\partial^2_BU^{n},U^n\right)+\frac{2k_n}{1+r_n}a\left(U^{n},U^n\right)
+\frac{2k_n}{1+r_n}\left(BU^{n},U^n\right)\\
&=\frac{2k_n}{1+r_n}\left(f^{n},U^n\right),~~ n=2,3,\dotsc,N.
\end{split}
\end{equation}
By simple calculations, the first term on the left-hand side becomes
\begin{equation}\label{eq3.6}
\begin{split}
\frac{2k_{n}}{1+r_{n}}(\bar\partial^2_BU^{n},U^{n})
=&\frac{k_{n}}{1+r_{n}}\bar\partial^2_B|U^{n}|^{2}+\frac{1+2r_{n}}{(1+r_{n})^2}|U^{n}-U^{n-1}|^2     \\
&-s_{n}^2|U^{n-1}-U^{n-2}|^2-2s_{n}^2(U^{n}-U^{n-1},U^{n-1}-U^{n-2})   \\
=&\frac{k_{n}}{1+r_{n}}\bar\partial^2_B|U^{n}|^{2}+\frac{1+2r_{n}-r_n^2}{(1+r_{n})^2}|U^{n}-U^{n-1}|^2\\
&-2s_{n}^2|U^{n-1}-U^{n-2}|^2+s_n^2|U^n-2U^{n-1}+U^{n-2}|^2.
\end{split}
\end{equation}
Now
\begin{eqnarray}\label{eq3.10a}
2|(f^n,U^n)|\le 2\|f^n\|_*\|U^n\|
\le 2\|f^n\|^2_*+\frac{1}{2}\|U^n\|^2,
\end{eqnarray}
and, in view of (\ref{eq2.3}), \begin{eqnarray}\label{eq7.19}
2|(BU^n,U^n)|\le 2\gamma(t^n)\|U^n\||U^n|\le 2\gamma^2(t^n)|U^n|^2+\frac{1}{2}\|U^n\|^2.
\end{eqnarray}
Substitute (\ref{eq3.6}), (\ref{eq3.10a}) and (\ref{eq7.19}) into (\ref{eq3.8a}) to obtain
\begin{eqnarray}\label{eq3.4}
&&\frac{k_{n}}{1+r_{n}}\bar\partial^2_B|U^{n}|^{2}+\frac{1+2r_{n}-r_n^2}{(1+r_{n})^2}|U^{n}-U^{n-1}|^2
-2s_{n}^2|U^{n-1}-U^{n-2}|^2+\frac{k_n}{1+r_n}\|U^{n}\|^2\nonumber\\
&\le&\frac{2k_n}{1+r_{n}}\|f^n\|_*^2+\frac{2\gamma^2(t^n)k_n}{1+r_{n}}|U^{n}|^2.
\end{eqnarray}
By summation, we obtain
\begin{eqnarray}\label{eq7.15}
&& \sum\limits_{j=2}^n\frac{k_{j}}{1+r_{j}}\bar\partial^2_B|U^{j}|^{2}
+\frac{1+2r_{n}-r_n^2}{(1+r_{n})^2}|U^{n}-U^{n-1}|^2\nonumber\\
&&+ \sum\limits_{j=2}^{n-1}\left(\frac{1+2r_{j}-r_j^2}{(1+r_{j})^2}-2s^2_{j+1}\right)|U^{j}-U^{j-1}|^2
-2s^2_{2}|U^{1}-U^{0}|^2+\sum\limits_{j=2}^n\frac{k_j}{1+r_j}\|U^{j}\|^2\nonumber\\
&\le&2\sum\limits_{j=2}^n\frac{k_j}{1+r_{j}}\|f^j\|_*^2
+2\sum\limits_{j=2}^n\frac{\gamma^2(t^j)k_j}{1+r_{j}}|U^{j}|^2.
\end{eqnarray}
We first consider the first term of the left-hand side of (\ref{eq7.15}) and obtain
\begin{eqnarray}\label{eq3.14a}
\sum\limits_{j=2}^n\frac{k_{j}}{1+r_{j}}\bar\partial^2_B|U^{j}|^{2}
&=&\frac{1+2r_n}{(1+r_n)^2}|U^n|^2-s_{n-1}^2|U^{n-1}|^2\nonumber\\
&&+\sum\limits_{j=2}^{n-2}\left[\frac{1+2r_j}{(1+r_j)^2}-1+s_{j+2}^2\right]|U^{j}|^2  \nonumber\\
&& -\frac{1+2r_3}{(1+r_3)^2}|U^1|^2+s_2^2|U^0|^2.
\end{eqnarray}
Using the mean value theorem, we can easily verify that there holds for some $\overline{r}$ between $r_j$ and $r_{j+2}$,
\begin{eqnarray*}
\begin{split}
\frac{1+2r_j}{(1+r_j)^2}-1+s_{j+2}^2
&=\frac{r_{j+2}^2}{(1+r_{j+2})^2}-\frac{r_{j}^2}{(1+r_{j})^2} =\frac{2\bar r}{(1+\bar r)^3}(r_{j+2}-r_j).
\end{split}
\end{eqnarray*}
Noting that the nonnegative function  $ \frac{2r}{(1+r)^3}$ attains its maximum value $8/27$ at $r=1/2$, we obtain the lower bound
\begin{eqnarray}\label{eq3.15a}
\frac{1+2r_j}{(1+r_j)^2}-1+s_{j+2}^2\geq - \frac{8}{27}[r_{j}-r_{j+2}]_+.
\end{eqnarray}
Substituting (\ref{eq3.15a}) into (\ref{eq3.14a}) yields
\begin{equation}\label{eq7.16}
\begin{split}
\sum\limits_{j=2}^n\frac{k_{j}}{1+r_{j}}\bar\partial^2_B|U^{j}|^{2}\geq&\frac{1+2R}{(1+R)^2}|U^{n}|^{2}
-\frac{R^2}{(1+R)^2}|U^{n-1}|^2 -\frac{1+2r_3}{(1+r_3)^2}|U^1|^2    \\
&+s_2^2|U^0|^2-\sum\limits_{j=2}^{n-2}\frac{8}{27}[r_{j}-r_{j+2}]_+|U^j|^2.
\end{split}
\end{equation}

Since $r_{\max}< R_0$, we have $1+2r_{j}-r_j^2> 0$ and $s_{j}^2< \frac{1}{2}, j=2,\dotsc,N$. As a consequence, it holds that $$\frac{1+2r_{j}-r_j^2}{(1+r_{j})^2}-2s_{j+1}^2=1-2s_j^2-2s_{j+1}^2\ge -2s_j^2.$$
Taking $\frac{1+2r_3}{(1+r_3)^2}\le 1$ into account, from (\ref{eq7.16}) and (\ref{eq7.15}) we obtain
\begin{equation}\label{eq7.20}
\begin{split}
 &\frac{1+2R}{(1+R)^2}|U^{n}|^{2}
-\frac{R^2}{(1+R)^2}|U^{n-1}|^2+\sum\limits_{j=2}^n\frac{k_j}{1+r_j}\|U^j\|^2\\
\le& |U^{1}|^2+2s^2_2|U^{1}-U^{0}|^2+\frac{8}{27}\sum\limits_{j=2}^{n-2}[r_j-r_{j+2}]_+|U^j|^2 \\
&+2\sum\limits_{j=2}^n\frac{k_j}{1+r_{j}}\|f^j\|_*^2+2\sum\limits_{j=2}^{n-1}s^2_{j}|U^j-U^{j-1}|^2
+2\sum\limits_{j=2}^n\frac{\gamma^2(t^j)k_j}{1+r_{j}}|U^{j}|^2.
\end{split}
\end{equation}
Now using {\sc Theorem \ref{lem3.1}} and $s_j<\frac{\sqrt{2}}{2}$, we have
\begin{equation}\label{eq3.18a}
\begin{split}
2\sum\limits_{j=2}^{n-1}s^2_{j}|U^j-U^{j-1}|^2=&2\sum\limits_{j=2}^{n-1}k_j^2s^2_{j}|\bar\partial U^j|^2\le\sqrt{2} k_{\max}\sum\limits_{j=2}^{n-1}k_js_{j}|\bar\partial U^j|^2\\
\le & C_2k_{\max}\left(\sum\limits_{j=2}^{n}k_j\left|f^{j}\right|^2
+k_2s_2\left|\bar\partial U^{1}\right|^2+\|U^1\|^2\right).
\end{split}
\end{equation}
Substitute (\ref{eq3.18a}) into (\ref{eq7.20}) to obtain
\begin{equation}\label{eq3.27}
\begin{split}
 \frac{1+2R}{(1+R)^2}|U^{n}|^{2}+\frac{1}{1+R}\sum\limits_{j=2}^nk_j\|U^j\|^2\le \frac{R^2}{(1+R)^2}|U^{n-1}|^2+Q_n,
\end{split}
\end{equation}
where
\begin{eqnarray*}
Q_n&=&|U^{1}|^2+2s^2_2|U^{1}-U^{0}|^2+\frac{8}{27}\sum\limits_{j=2}^{n-2}[r_j-r_{j+2}]_+|U^j|^2 +2\sum\limits_{j=2}^n\frac{k_j}{1+r_{j}}\|f^j\|_*^2\\
&&
+2\sum\limits_{j=2}^n\frac{\gamma^2(t^j)k_j}{1+r_{j}}|U^{j}|^2+C_2k_{\max}\left(\sum\limits_{j=2}^{n}k_j\left|f^{j}\right|^2
+k_2s_2\left|\bar\partial U^{1}\right|^2+\|U^1\|^2\right).
\end{eqnarray*}

The remaining part of this proof is analogous with that of Theorem 3 in \cite{Emmrich05}. We first show
\begin{eqnarray}\label{eq3.28}
|U^{n}|^{2}+\sum\limits_{j=2}^nk_j\|U^j\|^2\le c_3Q_n.
\end{eqnarray}
To do this, let $n^*=n^*(n)$ be such that $|U^{n^*}|=\max\limits_{l=1,\dotsc,n}|U^l|$ for $n=2,3,\dotsc,N$. We first note that
for $n^*=1$,
\begin{eqnarray}\label{eq3.29}
|U^{n^*}|^2\le \frac{(1+R)^2}{1+2R-R^2}Q_n.
\end{eqnarray}
Now we show that (\ref{eq3.29}) is valid for $n^*\ge 2$. Due to $Q_{n^*}\le Q_n$, it follows from (\ref{eq3.27}) with $n=n^*$ that
\begin{equation}\label{eq3.30}
\frac{1+2R}{(1+R)^2}|U^{n^*}|^{2}
\le \frac{R^2}{(1+R)^2}|U^{n^*-1}|^2+Q_{n^*}\le \frac{R^2}{(1+R)^2}|U^{n^*-1}|^2+Q_{n}.
\end{equation}
Since $R<1+\sqrt{2}$, (\ref{eq3.30}) implies (\ref{eq3.29}). As a results of (\ref{eq3.29}), we have
\begin{equation}
\begin{split}
\frac{1+2R}{(1+R)^2}|U^{n}|^{2}+\frac{1}{1+R}\sum\limits_{j=2}^nk_j\|U^j\|^2
\le \left(\frac{R^2}{1+2R-R^2}+1\right)Q_n=\frac{1+2R}{1+2R-R^2}Q_n.
\end{split}
\end{equation}
Then (\ref{eq3.28}) follows from $1<R< R_0$.

Now when $k_{\max}$ satisfies $2c_3\gamma^2k_{\max}\le c_2<1$, an application of Lemma \ref{lem2.3} to (\ref{eq3.28}) results in
\begin{equation}
\begin{split}
|U^{n}|^{2}+\sum\limits_{j=2}^nk_j\|U^j\|^2\le \frac{c_3C_3}{1-c_2}E_n f_U,
\end{split}
\end{equation}
and therefore (\ref{eq3.1}), where $C_3=\max\{2,C_2\}$ and
\begin{equation*}
\begin{split}
E_n=&\left[1+\frac{2\gamma^2k_{n-1}}{1-c_2}\right]\prod\limits_{j=2}^{n-2}
\left[1+\frac{8c_2}{27}[r_j-r_{j+2}]_++\frac{2\gamma^2(t^j)k_{j}}{1-c_2}\right]\\
\le& \exp\left(\frac{2c_3}{1-c_2}\left(\frac{4}{27}\Phi_n+\sum\limits_{j=2}^n\gamma^2(t^j)k_j\right)\right).
\end{split}
\end{equation*}
This completes the proof of {\sc Theorem \ref{th1}}.
\end{proof}

For the case $f\in L^2(J;V^*)$, we have the following result.
\begin{theorem}[$l^\infty(J;H)$ and $l^2(J;V)$ stability, $f\in L^2(J;V^*)$]\label{th3.3}
Let $f\in L^2(J;V^*)$ and $r_{\max}\le R$ with $1<R<R_0=\sqrt{2}+1\approx 2.414$. If there exist constants $\tilde c_1$ and $c_2$ such that $k_{\max}$ satisfies
$$(2+\sqrt{2})\gamma^2k_{\max}\le \tilde c_1<1$$ and \eqref{eq3.7}, then the following estimate holds for $n=2,3,\dotsc,N$:
\begin{equation}
\begin{split}
|U|^{2}_{l^\infty(t^2,t^n;H)}+\|U\|^2_{l^2(t^2,t^n;V)}\le C \tilde f_U,
\end{split}
\end{equation}
with
$$\tilde f_U=|U^{1}|^2+k_{\max} \|U^1\|^2+(k_1^2+k_{\max}k_2)|\bar\partial U^{1}|^2+\sum\limits_{j=2}^n(k_{\max}+k_j)\|f^j\|_*^2.$$
Here, $C$ depends on $\gamma$, $\tilde c_1$, $c_i$, $i=2,3$, and $R$, $T$, $\Phi_n$.
\end{theorem}
\begin{proof} When $f\in L^2(0,T;V^*)$, we can obtain a similar inequality to (\ref{eq3.4a})
\begin{equation}
\begin{split}
&k_ns_n\left(\left|\bar\partial U^n\right|^2-\left|\bar\partial U^{n-1}\right|^2+\left|\bar\partial U^n-\bar\partial U^{n-1}\right|^2\right)+2k_n\left|\bar\partial U^{n}\right|^2\\
&+\|U^{n}\|^2-\|U^{n-1}\|^2+\|U^{n}- U^{n-1}\|^2\\
\le&2\gamma(t^n) k_n\|U^n\|\left|\bar\partial U^{n}\right|+2k_n\left\|f^{n}\right\|_*\left\|\bar\partial U^n\right\|\\
\le&\frac{1}{\epsilon}\gamma^2(t^n)k_n\|U^n\|^2+\epsilon k_n\left|\bar\partial U^{n}\right|^2+\left\|f^{n}\right\|_*^2+ \|U^{n}- U^{n-1}\|^2.
\end{split}
\end{equation}
Then we have an estimate similar to (\ref{eq7.3})
\begin{equation}
\begin{split}
&k_n\left|\bar\partial U^n\right|^2
+\left\|U\right\|^2_{l^2(t^2,t^n;H,H)}+\|U\|^2_{l^\infty(t^2,t^n;V)}\\
\le&\tilde C_1\left(\sum\limits_{j=2}^{n}\left\|f^{j}\right\|_*^2
+k_2s_2\left|\bar\partial U^{1}\right|^2+\|U^1\|^2\right),
\end{split}
\end{equation}
where
$$\tilde C_1=\frac{2+\sqrt{2}}{1-\tilde c_1}\exp\left(\frac{2+\sqrt{2}}{1-\tilde c_1}\sum\limits_{j=2}^n\gamma^2(t^j)k_j\right)\le \frac{2+\sqrt{2}}{1-\tilde c_1}\exp\left(\frac{2+\sqrt{2}}{1-\tilde c_1}\gamma^2t^n\right).$$

The remaining part of this proof is analogous with that of Theorem \ref{th1} and so is omitted.
\end{proof}

We note that under the condition $R<R_1$, we cannot currently obtain the $l^\infty(J;H)$ and $l^2(J;V)$ stability of the variable step-size BDF2 method when $f\in L^2(J;V^*)$.

We conclude this section with a few remarks about our stability results.

Our first remark is that to avoid the complication arising from the term $s_n^2|U^n-2U^{n-1}+U^{n-2}|^2$ in (\ref{eq3.6}), we cancel this term in the proof of Theorem \ref{th1}. Following the proof of Theorem \ref{th1}, we can obtain, for $n=2,3,\dotsc,N$,
\begin{eqnarray*}
&&|U|^{2}_{l^\infty(0,t^n;H)}+\|U\|^2_{l^2(0,t^n;V)}+\sum\limits_{j=2}^ns_j^2|U^j-2U^{j-1}+U^{j-2}|^2\\
&\le& \begin{cases}
C f_U,~~ {\hbox{if}}~f\in L^2(J;H);\\
C \tilde f_U,~~ {\hbox{if}}~f\in L^2(J;V^*).
\end{cases}
\end{eqnarray*}

It is noteworthy that the larger the admissible step-size $k_{\max}$, the smaller the admissible step-size ratio $r_{\max}$ will be. Especially, $k_{\max}\to 0$ as $r_{\max}\to R_0$. The relationship can be observed from the condition (\ref{eq3.7}). We also note that if $\gamma=0$, then the variable step-sizes BDF2 method with $r_{\max}<R_0$ is unconditionally stable. Further, if there exists a constant $\alpha>0$ such that
\begin{eqnarray}\label{eq2.1}
|v|\le \alpha\|v\|,\quad \forall v\in V,
\end{eqnarray}
then when $\alpha\gamma^2<1$, the conditions upon $k_{\max}$ can be relaxed to $c\gamma^2k_{\max}<1$, no longer dependent on the step-size ratio $r_{\max}$. This is because in this case the inequalities (\ref{eq3.10a}) and (\ref{eq7.19}) in the proof of  Theorem \ref{th1} can be replaced by
\begin{eqnarray}
2|(f^n,U^n)|\le \frac{1}{2(1-\alpha\gamma^2(t^n)-\epsilon)}\|f^n\|^2_*+2(1-\alpha\gamma^2(t^n)-\epsilon)\|U^n\|^2,\quad \epsilon\in (0,1-\alpha\gamma^2),\nonumber\\
\end{eqnarray}
and \begin{eqnarray}
2|(BU^n,U^n)|\le 2\gamma(t^n)\|U^n\||U^n|\le 2\alpha\gamma^2(t^n)\|U^n\|^2,
\end{eqnarray}
respectively.

The third remark is about the operators $A$ and $B$. It is natural to write $\tilde A=A+B$ and consider the linear problems $u^{\prime}(t)+\tilde Au(t)=f(t)$. In this case, the operator $\tilde A$ will satisfy a G\aa rding inequality
\begin{equation}
(\tilde Au,u)\ge \tilde \alpha\|u\|^2-\beta|u|^2,\qquad u\in V.
\end{equation} By a standard change of variables the equation can be equivalently written in a form such that the new operator is coercive. Then we can obtain similar results to those in {\sc Theorems 3.1--3.4}.

Finally, we think that the value of $R_0$ for the $l^\infty(J;H)$ stability cannot be improved when dealing with arbitrary variable step-sizes, as said for the zero-stability (see, for example, \cite{Grigorieff83,Calvo90,Calvo93}).

\section{Error estimates} In this section, based on the stability estimates (\ref{eq7.3}), \eqref{eq3.8aa}, \eqref{eq3.3aa}, and (\ref{eq3.1}), we derive a priori error estimates for the variable step-size BDF2 method (\ref{eq2.5}). To do this, we first consider the consistency error $d^n_2$ of the method (\ref{eq2.5}) for the solution $u$ of (1.1), which is given by
\begin{equation}
d^n_2=\bar\partial^2_B u(t^n)+Au(t^n)+Bu(t^n)-f(t^n)=\bar\partial^2_B u(t^n)-u^\prime(t^n),\quad n=2,\dotsc,N.
\end{equation}
By Taylor expanding about $t^{n-1}$, for $n\ge 2$, we obtain,
\begin{equation}
d^n_2=\frac{(1+r_n)}{2k_n}\int^{t^n}_{t^{n-1}}(t^{n-1}-t)^2u^{\prime\prime\prime}(t)dt
-\frac{r_ns_n}{2k_{n}}\int^{t^{n}}_{t^{n-2}}(t-t^{n-2})^2u^{\prime\prime\prime}(t)dt.
\end{equation}

As mentioned in Section 2, $\bar\partial^2_B$ degenerate to $\bar\partial^1_B$ whenever $r_n=0$. In this case, we come up with
\begin{equation}
d^n_1=\bar\partial^1_B u^n-u^\prime(t^n)=-\frac{1}{k_n}\int^{t^n}_{t^{n-1}}(t-t^{n-1})u^{\prime\prime}(t)dt.
\end{equation}
Then the consistency errors of these schemes can be bounded by the following
\begin{equation}\label{eq3.47}
|d^n_2|\le c_4k_n^2,\quad n\ge 2;\qquad |d^n_1|\le c_4k_n,\quad n\ge 1,
\end{equation}
where $c_4$ depends only on some derivatives of the exact solution $u$.

\subsection{$l^\infty(J;V)$ and $l^2(J;H,H)$ error estimates} Let $e^n=u(t^n)-U^n~(n=0,1,\dotsc,N)$ be the error. We first derive the error bound in the $l^\infty(J;V)$ and $l^2(J;H,H)$ norms.
\begin{theorem}[$l^\infty(J;V)$ and $l^2(J;H,H)$ error estimates under $R<R_0$]\label{th4.1} Let $r_{\max}\le R$ with $1<R<R_0=\sqrt{2}+1\approx 2.414$. If $e^0=0$ and there exists a constant $c_1$ such that $k_{\max}$ satisfies \eqref{eq7.2}, then the error $e^n =u(t^n)-U^n~ (n=2,3,\dotsc,N)$ satisfies
\begin{equation}\label{eq4.5}
\left\|e\right\|^2_{l^2(t^2,t^n;H,H)}+\|e\|^2_{l^\infty(t^2,t^n;V)}\le C\left(k_{\max}^4t^n+\frac{k_2}{k_1^2}\left|e^{1}\right|^2+\|e^1\|^2\right),
\end{equation}
where $C$ depends on $t^n$ and the constants $\gamma$, $c_1$, $c_4$.
\end{theorem}

\begin{proof} It follows from (1.1) and (\ref{eq2.5}) that
\begin{equation}\label{eq4.6}
\bar\partial^2_Be^n+Ae^n+Be^n=d^n_2,\quad n\ge 2.
\end{equation}
Then we may apply Theorem 3.1 to obtain
\begin{equation}
\left\|e\right\|^2_{l^2(t^2,t^n;H,H)}+\|e\|^2_{l^\infty(t^2,t^n;V)}
\le C\left(\sum\limits_{j=2}^{n}k_j\left|d_2^{j}\right|^2
+k_2s_2\left|\bar\partial e^{1}\right|^2+\|e^1\|^2\right).
\end{equation}
Noting $e^0=0$ and using the consistency errors bound (\ref{eq3.47}), we obtain (\ref{eq4.5}) and complete the proof.
\end{proof}


Now we estimate the starting error $e^1$. For this purpose, we consider the consistency error $d^1_2$ of the first step by the trapezoidal scheme
 \begin{equation}
d^{1}_2:=\bar{\partial}u^{1}+Au^{\frac{1}{2}}+Bu^{\frac{1}{2}}-f^{\frac{1}{2}}.
\end{equation}
It is well known that, under obvious regularity assumptions,
\begin{equation}
\| d^{1}_2\|_{*}\le c_4 k^{2}_{1}.
\end{equation}
When
\begin{eqnarray}\label{eq4.10}
k_1< \frac{3}{\gamma^2(t^1)},
\end{eqnarray}from $e^0=0$ and
 \begin{equation*}
e^{1}+\frac{k_{1}}{2}[Ae^{1}+Be^{1}]=k_{1}d^{1}_2,
\end{equation*}
we obtain
\begin{equation}\label{eq4.11}
|e^{1}|^{2}+k_{1}\| e^{1}\|^{2}\le 8 c^2_4k_{1}^{5}.
\end{equation}
Then we have the following corollary.
\begin{corollary} Let $r_{\max}\le R$ with $1<R<R_0=\sqrt{2}+1\approx 2.414$, and let the starting value $U^1$ be computed by the trapezoidal scheme \eqref{eq2.4} with $k_1$ satisfying \eqref{eq4.10}. If $e^0=0$ and there exists a constant $c_1$ such that $k_{\max}$ satisfies \eqref{eq7.2}, then the error $e^n=u(t^n)-U^n~(n=2,3,\dotsc,N)$ satisfies
\begin{equation}
\left\|e\right\|^2_{l^2(t^1,t^n;H,H)}+\|e\|^2_{l^\infty(t^1,t^n;V)}\le C( k_{\max}^4+k_1^4),
\end{equation}
where $C$ depends only on $C_1$, $t^n$, $c_4$ and $r_2$.
\end{corollary}

From Corollary 4.2, we know that the optimal convergence order of the variable step-size BDF2 can be achieved in the $l^\infty(J;V)$ and $l^2(J;H,H)$ norms when the trapezoidal scheme (\ref{eq2.4}) is used to compute the starting value $U^1$. We also notice that from Theorem 3.2 where the condition on the step-size rations has been relaxed to $\frac{3+\sqrt{17}}{2}$, we can also obtain the second order convergence result in the $l^\infty(J;V)$ and $l^2(J;V)$ norms if the trapezoidal scheme (\ref{eq2.4}) is used to compute the starting value $U^1$. The following theorem states this fact.

\begin{theorem}[$l^\infty(J;V)$ and $l^2(J;V)$ error estimates under $R<R_1$] Let $r_{\max}\le R$ with $1<R<R_1=(3+\sqrt{17})/2\approx 3.561$. If $e^0=0$ and there exists a constant $c_1$ such that $k_{\max}$ satisfies \eqref{eq3.7a}, then the error $e^n=u(t^n)-U^n~(n=2,3,\dotsc,N)$ satisfies
\begin{equation}\label{eq4.13}
\left\|e\right\|^2_{l^2(t^2,t^n;V)}+\|e\|^2_{l^\infty(t^2,t^n;V)}\le C\left(k_{\max}^4t^n+\frac{1}{k_1}\left|e^{1}\right|^2+\|e^1\|^2\right).
\end{equation}
Furthermore, if the starting value $U^1$ is computed by the trapezoidal scheme \eqref{eq2.4} with $k_1$ satisfying \eqref{eq4.10}, then we have \begin{equation}
\left\|e\right\|^2_{l^2(t^1,t^n;V)}+\|e\|^2_{l^\infty(t^1,t^n;V)}\le C( k_{\max}^4+k_1^4).
\end{equation}
Here $C$ depends on $t^n$ and the constants $\gamma$, $c_1$, $c_R$, $c_4$, $C_2$.
\end{theorem}

\subsection{$l^\infty(J;H)$ and $l^2(J;V)$ error estimates} This subsection is devoted to the $l^\infty(J;H)$ and $l^2(J;V)$ error estimates. We have the following results.

\begin{theorem}[$l^\infty(J;H)$ and $l^2(J;V)$ error estimates under $R<R_0$]\label{th3.5} Let $r_{\max}\le R$ with $1<R<R_0=\sqrt{2}+1\approx 2.414$. If $e^0=0$ and there exist constants $c_1$ and $c_2$ such that $k_{\max}$ satisfies \eqref{eq7.2} and \eqref{eq3.7}, then the error $e^n=u(t^n)-U^n~(n=2,3,\dotsc,N)$ satisfies
\begin{equation}\label{eq2.12}
\begin{split}
|e|^2_{l^\infty(t^2,t^n;H)}+\|e\|^2_{l^2(t^2,t^n;V)}\le& C\left(|e^{1}|^2+k_{\max} \|e^1\|^2+\frac{k_{\max}k_2}{k_1^2}|e^{1}|^2+k^4_{\max}\right),
\end{split}
\end{equation}
where $C$ depends only on $\gamma$, $c_i$, $i=1,2,3,4$, and $R$, $t^n$, $\Phi_n$.
\end{theorem}

\begin{proof} An application of Theorem \ref{th1} to (\ref{eq4.6}) yields
\begin{equation}
\begin{split}
&|e|^2_{l^\infty(t^2,t^n;H)}+\|e\|^2_{l^2(t^2,t^n;V)}\\
\le& C\left(|e^{1}|^2+k_{\max} \|e^1\|^2+(k_1^2+k_{\max}k_2)|\bar\partial e^{1}|^2+\sum\limits_{j=2}^nk_j\left(\|d^j_2\|_*^2+k_{\max}\left|d^{j}_2\right|^2\right)\right).
\end{split}
\end{equation}
Taking $e^0=0$ and the consistency errors bound (\ref{eq3.47}) into account, we obtain (\ref{eq2.12}) and complete the proof.
\end{proof}

A comparison with the estimate \eqref{eq4.13} in Theorem 4.3 suggests that the $l^2(J;V)$ estimate obtained under the condition $R<R_0$ is sharper than the estimate obtained under the condition $R<R_1$.

Combining Theorem \ref{th3.5} and the estimates (\ref{eq4.11}) for the starting error $e^1$ produced by the trapezoidal scheme (\ref{eq2.4}) leads to the following corollary.
\begin{corollary} Let $r_{\max}\le R$ with $1<R<R_0=\sqrt{2}+1\approx 2.414$, and let the starting value $U^1$ be computed by the trapezoidal scheme \eqref{eq2.4} with $k_1$ satisfying \eqref{eq4.10}. If $e^0=0$ and there exist constants $c_1$ and $c_2$ such that $k_{\max}$ satisfies \eqref{eq7.2} and \eqref{eq3.7}, then the error $e^n=u(t^n)-U^n~(n=2,3,\dotsc,N)$ satisfies
\begin{equation}
\begin{split}
|e|^2_{l^\infty(t^2,t^n;H)}+\|e\|^2_{l^2(t^2,t^n;V)}\le& C\left(k_1^5+k_{\max} k_1^4 +k^4_{\max}\right),
\end{split}
\end{equation}
where $C$ depends only on $\gamma$, $c_i$, $i=1,2,3,4$, and $R$, $t^n$, $\Phi_n$.
\end{corollary}

This corollary means that the variable step-size BDF2 method can achieve optimal convergence order in the $l^\infty(J;H)$ and $l^2(J;V)$ norms when the trapezoidal scheme (\ref{eq2.4}) is used to compute the starting value $U^1$.

\subsection{The backward Euler method for the starting value}In this subsection, we consider the convergence order of the variable step-size BDF2 method if the backward Euler method is used to calculate the starting value $U^1$. To do this, we need the following a priori estimate for $e^1$ with smooth $u$ (see, e.g., \cite{Becker98,Thomee97})
\begin{eqnarray}\label{eq6.1}
|e^1|\le c_5k_1^2,
\end{eqnarray}
where $c_5$ depends on the derivatives of the exact solution $u$. It follows from (\ref{eq2.9}) that
\begin{eqnarray}\label{eq6.3}
\bar\partial e^1+Ae^1+Be^1=\bar\partial u^1-u^\prime(t^1).
\end{eqnarray}
Taking in (\ref{eq6.3}) the inner product with $k_1e^1$, we obtain
\begin{eqnarray*}
|e^1|^2+k_1\|e^1\|^2\le |e^0|^2+\gamma(t^1) k_1\|e^1\||e^1|+k_1|e^1||\bar\partial u^1-u^\prime(t^1)|.
\end{eqnarray*}
Using $e^0=0$, (\ref{eq6.1}) and
$$|k_1[\bar\partial u^1-u^\prime(t^1)]|=\left|\int^{t^1}_0su^{\prime\prime}(s)ds\right|\le c_6k_1^2,$$
if $\gamma^2(t^1) k_1<1$, we have
\begin{eqnarray}\label{eq6.4}
|e^1|^2+k_1\|e^1\|^2\le 4c_6^2 k_1^4.
\end{eqnarray}
Then we have the following corollary.
\begin{corollary}[Backward Euler method for $U^1$] Let $U^n$, $n=2,3,\dotsc,N$, be the solution sequence of (\ref{eq5.4}), $U^0=u^0$, and let the starting value $U^1$ be computed by the backward Euler method \eqref{eq2.9} with $k_1$ satisfying $\gamma^2(t^1) k_1<1$.

{\rm (i)} If $r_{\max}\le R$ with $1<R<R_1=(3+\sqrt{17})/2\approx 3.561$, and if there exists a constant $c_1$ such that $k_{\max}$ satisfies \eqref{eq3.7a}, then we have following estimate
\begin{equation}
\left\|e\right\|^2_{l^2(t^1,t^n;V)}+\|e\|^2_{l^\infty(t^1,t^n;V)}\le C( k_{\max}^4+k_1^3),
\end{equation}
where $C$ depends only on $C_2$, $t^n$, $c_4$, $c_5$, $c_6$ and $c_R$;

{\rm (ii)} If we further restrict $r_{\max}\le R$ with $1<R<R_0=\sqrt{2}+1\approx 2.414$, and if there exists a constant $c_1$ such that $k_{\max}$ satisfies \eqref{eq7.2}, then we have following estimate
\begin{equation}
\left\|e\right\|^2_{l^2(t^1,t^n;H,H)}+\|e\|^2_{l^\infty(t^1,t^n;V)}\le C( k_{\max}^4+k_1^3),
\end{equation}
where $C$ depends only on $C_1$, $t^n$, $c_4$, $c_5$, $c_6$ and $r_2$;

{\rm (iii)} Let $r_{\max}\le R$ with $1<R<R_0=\sqrt{2}+1\approx 2.414$. If there exist constants $c_1$ and $c_2$ such that $k_{\max}$ satisfies \eqref{eq7.2} and \eqref{eq3.7}, then we have
\begin{equation}
\begin{split}
|e|^2_{l^\infty(t^1,t^n;H)}+\|e\|^2_{l^2(t^1,t^n;V)}\le& C\left(k_1^4+k^4_{\max}\right),
\end{split}
\end{equation}
where $C$ depends only on $\gamma$, $c_i$, $i=1,\dotsc,6$, and $R$, $t^n$, $\Phi_n$.
\end{corollary}
\begin{proof}
Substitute (\ref{eq6.4}) into \eqref{eq4.13}, (\ref{eq4.5}) and (\ref{eq2.12}) to obtain the required results.
\end{proof}

This corollary reveals that the convergence order of the constant step-size BDF2 method is only $1.5$ with respect to $k_1$ in the $l^\infty(J;V)$ and $l^2(J;H,H)$ norms when the backward Euler method (\ref{eq2.9}) is used to compute the starting value $U^1$. This implies that the order reduction phenomenon may appear for the constant step-size BDF2 method in the $l^\infty(J;V)$ and $l^2(J;H,H)$ norms if the backward Euler method (\ref{eq2.9}) is used to compute the starting value $U^1$. However, if we choose $k_1=O(k_{\max}^{4/3})$, the variable step-size BDF2 method can achieve optimal second order of convergence, even if the starting value $U^1$ is computed by the backward Euler method (\ref{eq2.9}). These are observed in the following numerical experiments.

\section{Variable step-size BDF2 method for semilinear parabolic equations} In this section, we derive error bounds of the variable step-size BDF2 method for the semilinear parabolic equation (\ref{eq1.1a}). Applying the variable step-size BDF2 method to (\ref{eq1.1a}) yields
\begin{eqnarray}\label{eq5.4}
\bar\partial^2_BU^{n}+AU^{n}=f(t^{n},U^{n}),\qquad n=2,3,\dotsc,N.
\end{eqnarray}

Let $\mathcal B_{u(t)}:=\left\{v\in V:\|v-u(t)\|\le 1\right\}$, i.e., a ball of radius $1$ centred at the value $u(t)$ of the solution $u$ at time $t$. We assume that $f(t,\cdot)$ satisfies the following local Lipschitz condition in a ball $\mathcal B_{u(t)}$ (see, e.g., \cite{Akrivis15}),
\begin{eqnarray}\label{eq5.2a}
 |f(t,v)-f(t,w)|\le \gamma(t) \|v-w\|,\qquad \forall v,w\in \mathcal B_{u(t)},~~t\in J,
 \end{eqnarray}
with a smooth nonnegative function $\gamma: J\to \mathbb R$. In view of the condition (\ref{eq5.2a}), proceeding as in the proof of Theorems 3.1, 3.2 and 3.3, we have the following error estimates for the variable step-size BDF2 method (\ref{eq5.4}).

\begin{theorem}\label{th5.3} Let $U^n$, $n=2,3,\dotsc,N$, be the solution sequence of (\ref{eq5.4}), and let $U^0=u^0$.

{\rm (i)} If $r_{\max}\le R$ with $1<R<R_1=(3+\sqrt{17})/2\approx 3.561$, and if there exists a constant $c_1$ such that $k_{\max}$ satisfies \eqref{eq3.7a}, then we have following estimate
\begin{equation}
\left\|e\right\|^2_{l^2(t^2,t^n;V)}+\|e\|^2_{l^\infty(t^2,t^n;V)}\le C\left(k_{\max}^4t^n+\frac{1}{k_1}\left|e^{1}\right|^2+\|e^1\|^2\right),
\end{equation}
where $C$ depends only on $C_2$, $t^n$, $c_4$, and $c_R$;

{\rm (ii)} If we further restrict $r_{\max}\le R$ with $1<R<R_0=\sqrt{2}+1\approx 2.414$, and if there exists a constant $c_1$ such that $k_{\max}$ satisfies \eqref{eq7.2}, then we have following estimate
\begin{equation}
\left\|e\right\|^2_{l^2(t^2,t^n;H,H)}+\|e\|^2_{l^\infty(t^2,t^n;V)}\le C\left(k_{\max}^4t^n+\frac{k_2}{k_1^2}\left|e^{1}\right|^2+\|e^1\|^2\right),
\end{equation}
where $C$ depends only on $C_1$, $t^n$, $c_4$, and $r_2$;

{\rm (iii)} Let $r_{\max}\le R$ with $1<R<R_0=\sqrt{2}+1\approx 2.414$. If there exist constants $c_1$ and $c_2$ such that $k_{\max}$ satisfies \eqref{eq7.2} and \eqref{eq3.7}, then we have
\begin{equation}
\begin{split}
|e|^2_{l^\infty(t^2,t^n;H)}+\|e\|^2_{l^2(t^2,t^n;V)}\le& C\left(|e^{1}|^2+k_{\max} \|e^1\|^2+\frac{k_{\max}k_2}{k_1^2}|e^{1}|^2+k^4_{\max}\right),
\end{split}
\end{equation}
where $C$ depends only on $\gamma$, $c_i$, $i=1,2,3,4$, and $R$, $t^n$, $\Phi_n$.
\end{theorem}

\begin{proof} We commence with the error equation. Subtracting (\ref{eq1.1a}) and (\ref{eq5.4}), we obtain
\begin{eqnarray}\label{eq5.12}
\bar\partial^2_Be^{n}+Ae^n=f(t^{n},u(t^{n}))-f(t^{n},U^{n})+d^n_2,\qquad n=2,3,\dotsc,N,
\end{eqnarray}
where $d^2_n$ is the consistency error of the method for the solution $u$ of (\ref{eq1.1a}) given by
\begin{equation}
d^n_2=\bar\partial^2_B u(t^n)+Au(t^n)-f(t^n,u(t^n))=\bar\partial^2_B u(t^n)-u^\prime(t^n),\quad n=2,\dotsc,N.
\end{equation}
It is easy to show that $|d^n_2|\le c_4k_n^2,\quad n\ge 2$. The remaining part of this proof is quite similar to the proof given earlier for linear problem and so is omitted.
\end{proof}

Based on Theorem 5.1, we can obtain similar error estimates to Corollaries 4.2, 4.5 and 4.6 when the trapezoidal scheme and the backward Euler scheme are used to compute the starting value $U^1$, respectively. We do not intend to state these results, for they are similar. We want to emphasize that the convergence order of the constant step-size BDF2 method may be reduced in the $l^\infty(J;V)$ and $l^2(J;H,H)$ norms if the backward Euler scheme is used to compute the starting value $U^1$.

\section{Numerical experiments} To support the analysis developed in this paper, in this section, we present some numerical examples. We proceed by studying two different cases. The first one concerns a linear parabolic equation, while in the second one we consider a $2$D semilinear case.

\subsection{Experiment 1: linear case} Let us consider the following linear parabolic equation
\begin{eqnarray}\label{eq6.1a}
 u_t(t,x)=u_{xx}(t,x)+bu+f(t,x), \quad t\in [0,4],\quad x\in [0,1].
\end{eqnarray}
The function $f(t,x)$ and the initial and boundary values are selected in such a way that the exact solution becomes
$$u(t,x)=x(1-x)\exp(-t).$$
The space derivative $u_{xx} $ of (\ref{eq6.1a}) will be approximated with central finite difference of second order. After spatial discretization a system of ODEs results,
\begin{eqnarray} u_i^\prime(t)&=&\Delta
x^{-2}[u_{i-1}(t)-2u_{i}(t)+u_{i+1}(t)]+b u_i(t)+f_i(t), \qquad t\ge 0,\\
u_i(0)&=&i\Delta x(1- i\Delta x),~~i=1,2,\dotsc,M-1,\\
u_0(t)&=&u_{M}(t)=0,~~t\ge 0,
\end{eqnarray}
where $\Delta x=1/M$, $x_i=i\Delta x$, $u_i(t)$ is meant to
approximate the solution of (\ref{eq6.1a}) at the point $(t,x_i)$, and $f_i$ stands for $f(t,x_i)$.

We first test the stability of the variable step-sizes BDF2 (VSBDF2 for short) method under the condition $r_{\max}\le R$ with $1<R<R_0=\sqrt{2}+1\approx 2.414$. To do this, we consider an extreme case, a geometric mesh with $r_n=r=2.4<\sqrt{2}+1$ and $k_1= T(r-1)/(r^N-1)$. The starting value $U^1$ is computed by the backward Euler (BE for short) method. The time evolution of the discrete $L^2$ norm $|U^n|_{l^2}=\left(\sum\limits_{i=1}^M\Delta x|U^n_i|^2\right)^{1/2}$ of the numerical solution $U^n$ is presented in Fig. 6.1.
\begin{figure}
\centering
\includegraphics[height=6cm,width=8cm]{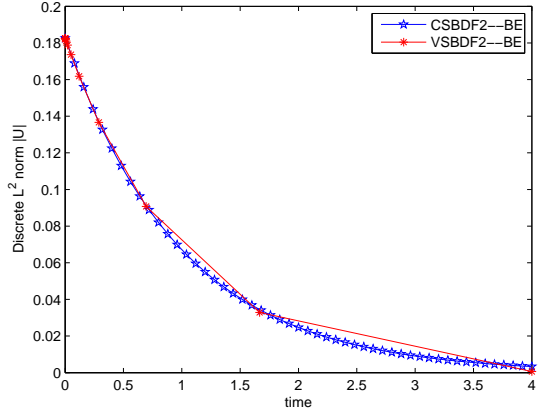}
\caption{Time evolution of the discrete $L^2$ norm $|U^n|_{l^2}$ of the numerical solutions obtained by the CSBDF2 method with the starting backward Euler step (CSBDF2-BE) and the VSBDF2 method ($r_{\max}=r=2.4$) with the starting backward Euler step (VSBDF2-BE) for linear problem (\ref{eq6.1a}), where $N=50$, $M=100$. }
\label{fig:1} 
\centering
\end{figure}
A comparison with the results of the constant step-size BDF2 (CSBDF2 for short) method reveals that the VSBDF2 method with $r_n=r=2.4$ is also stable.

Let $e^n_i=U^n_i-u(t^n,x_i)$ with the numerical solution $U^n_i$ approximating the exact solution $u(t^n,x_i)$ at point $(t^n,x_i)$. Since in this problem, there is not space discretization error, we set $M=100$. The time evolution of the discrete $L^2$ errors of the two methods (in the VSBDF2 method, $r_{\max}=r=1.1)$ are presented in Fig. 6.2. From Fig. 6.2 we observe that the VSBDF2 method is more efficient than the CSBDF2 method for this class of problems.
\begin{figure}
\centering
\includegraphics[height=6cm,width=8cm]{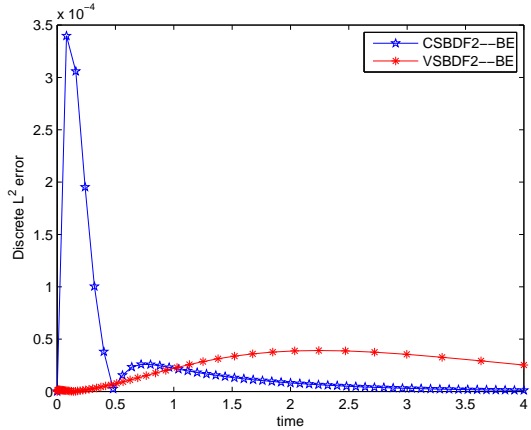}
\caption{Time evolution of the discrete $L^2$ errors of the CSBDF2 method with the starting backward Euler step (CSBDF2-BE) and the VSBDF2 method ($r_{\max}=r=1.1$) with the starting backward Euler step (VSBDF2-BE) for linear problem (\ref{eq6.1a}), where $N=50$, $M=100$. }
\label{fig:2} 
\centering
\end{figure}

To precisely test the convergence order of the methods in $l^\infty(J;V)$, $l^2(J;H,H)$, $l^\infty(J;H)$, and $l^2(J;V)$ norms, we first consider their discrete counterparts. Let the discrete $H^1$ norm $\|e^n\|_{\mathcal H^1}$ of the errors $e^n=[e^n_0,e^n_1,\dotsc, e^n_M]^T$ be calculated by $\|e^n\|_{\mathcal H^1}=\left(\sum\limits_{i=1}^M\Delta x|(e^n_i-e^n_{i-1})/\Delta x|^2\right)^{1/2}$. Then the discrete $l^\infty(J;V)$, $l^2(J;H,H)$, $l^\infty(J;H)$, and $l^2(J;V)$ errors are calculated by
\begin{eqnarray}
E_{l^\infty(J;V)}&=&\max\limits_{1\le n\le N}\|e^n\|_{\mathcal H^1},\quad E_{l^2(J;H,H)}=\left(\sum\limits_{n=2}^Nk_ns_n\sum\limits_{i=1}^M\Delta x\left|\frac{e^{n}_i-e^{n-1}_{i}}{k_n}\right|^2\right)^{1/2},\\
E_{l^\infty(J;H)}&=&\max\limits_{1\le n\le N}|e^n|_{l^2},\qquad
E_{l^2(J;V)}=\left(\sum\limits_{n=2}^Nk_n\|e^n\|_{\mathcal H^1}^2\right)^{1/2},
\end{eqnarray}
respectively.

We consider a mesh introduced by Becker in \cite{Becker98}: Choose the time levels $t^n$ according to $t^n=T(n/N)^{\varpi}$ with $\varpi\ge 1$. Note that $\varpi=1$ corresponds to constant step-size. In \cite{Becker98}, it has been observed that $\Phi_N$ is bounded since $r_n$ is decreasing. We also notice that $k_1=T(1/N)^{\varpi}$ and $k_{\max}=T\left[1-\left(\frac{N-1}{N}\right)^\varpi\right]\le T\varpi N^{-1}$. Obviously, when $\varpi\ge \frac{4}{3}$, we have $k_1=O(k_{\max}^{4/3})$. Then the VSBDF2 method with $\varpi\ge \frac{4}{3}$ can theoretically achieve optimal second order of convergence, whether the BE method (\ref{eq2.9}) or the trapezoidal formula (\ref{eq2.4}) (TF for short) is considered for computing the starting value $U^1$. The discrete errors $E_{l^\infty(J;V)}$,  $E_{l^2(J;H,H)}$, $E_{l^\infty(J;H)}$, and $E_{l^2(J;V)}$, and the convergence orders of the two methods, the CSBDF2 method and VSBDF2 method with $\varpi=3$, in different norms are listed in Tables 6.1, 6.2, 6.3 and 6.4, respectively. From Tables 6.1, 6.2, 6.3 and 6.4, we observe that all quantities of VSBDF2 are of optimal order two, whether the starting value $U^1$ is computed by the BE method (\ref{eq2.9}) or by the TF (\ref{eq2.4}). In fact, we find that all numerical data of VSBDF2 with the starting BE step (\ref{eq2.9}) (VSBDF2-BE) are almost the same as those of VSBDF2 with the starting TF step (\ref{eq2.4}) (VSBDF2-TF). This arises mainly because with our VSBDF2 method, $k_1=O(N^{-2}k_{\max})$ and therefore the errors $e^1$ has little effect on the global errors.

\begin{table}[htbp]
\centering\small
\caption{
The errors $E_{l^\infty(J;V)}$ and the convergence orders of the constant step-size BDF2 (CSBDF2) method and the variable step-size BDF2 (VSBDF2) method with $\varpi=3$ for problem (\ref{eq6.1a}). Upper: The backward Euler method for the starting value $U^1$; Bottom: The trapezoidal formula for the starting value $U^1$.}
\begin{tabular}{|c|cc|cc|}\hline
 &  CSBDF2&   & VSBDF2 &($\varpi=3$)\\
\cline{2-3}\cline{4-5}
$N$     & Error      &Order  & Error&Order   \\\hline
20           & 3.9026E-03   &              & 5.2723E-04   &               \\
40           & 1.5010E-03   & 1.3785       & 1.3204E-04   & 1.9975        \\
80           & 5.1347E-04   & 1.5476       & 3.3061E-05   & 1.9978       \\
160          & 1.7183E-04   & 1.5793       & 8.2644E-06   & 2.0001       \\
320          & 5.0390E-05   & 1.7698       & 2.0661E-06   & 2.0000        \\
   \hline
20           & 6.2112E-04   &              & 5.2723E-04   &              \\
40           & 1.6936E-04   & 1.8748       & 1.3204E-04   & 1.9975       \\
80           & 4.4067E-05   & 1.9423       & 3.3061E-05   & 1.9978       \\
160          & 1.1139E-05   & 1.9841       & 8.2644E-06   & 2.0001       \\
320          & 2.7969E-06   & 1.9937       & 2.0661E-06   & 2.0000       \\

   \hline
\end{tabular}
\end{table}\vspace{1mm}

\begin{table}[htbp]
\centering\small
\caption{
The errors $E_{l^2(J;H,H)}$ and the convergence orders of the constant step-size BDF2 (CSBDF2) method and the variable step-size BDF2 (VSBDF2) method with $\varpi=3$ for problem (\ref{eq6.1a}). Upper: The backward Euler method for the starting value $U^1$; Bottom: The trapezoidal formula for the starting value $U^1$.}
\begin{tabular}{|c|cc|cc|}\hline
 &  CSBDF2&   & VSBDF2 &($\varpi=3$)\\
\cline{2-3}\cline{4-5}
$N$     & Error      &Order  & Error&Order   \\\hline
20           & 1.2944E-03   &              & 1.2050E-04   &               \\
40           & 5.5973E-04   & 1.2095       & 3.0922E-05   & 1.9623        \\
80           & 2.2095E-04   & 1.3410       & 7.8074E-06   & 1.9857       \\
160          & 8.3180E-05   & 1.4094       & 1.9601E-06   & 1.9939       \\
320          & 3.0512E-05   & 1.4469       & 4.9096E-07   & 1.9972        \\
   \hline
20           & 1.9409E-04   &              & 1.2038E-04   &              \\
40           & 6.6528E-05   & 1.5447       & 3.0920E-05   & 1.9610       \\
80           & 2.0180E-05   & 1.7210       & 7.8074E-06   & 1.9856       \\
160          & 5.6362E-06   & 1.8401       & 1.9601E-06   & 1.9939       \\
320          & 1.4964E-06   & 1.9132       & 4.9096E-07   & 1.9972       \\

   \hline
\end{tabular}
\end{table}\vspace{1mm}

\begin{table}[htbp]
\centering\small
\caption{
The errors $E_{l^\infty(J;H)}$ and the convergence orders of the constant step-size BDF2 (CSBDF2) method and the variable step-size BDF2 (VSBDF2) method with $\varpi=3$ for problem (\ref{eq6.1a}). Upper: The backward Euler method for the starting value $U^1$; Bottom: The trapezoidal formula for the starting value $U^1$.}
\begin{tabular}{|c|cc|cc|}\hline
 &  CSBDF2&   & VSBDF2 &($\varpi=3$)\\
\cline{2-3}\cline{4-5}
$N$     & Error      &Order  & Error&Order   \\\hline
20           & 1.2422E-03   &              & 1.6782E-04
  &               \\
40           & 4.7773E-04   & 1.3786       & 4.2028E-05
  & 1.9975       \\
80           & 1.6345E-04   & 1.5473       & 1.0524E-05
 & 1.9977      \\
160          & 5.4688E-05   & 1.5796       & 2.6306E-06
 & 2.0002      \\
320          & 1.6031E-05   & 1.7704       & 6.5767E-07
  & 2.0000        \\
   \hline
20           & 1.9771E-04   &              & 1.6782E-04   &              \\
40           & 5.3908E-05   & 1.8748       & 4.2028E-05
& 1.9975        \\
80           & 1.4027E-05   & 1.9423       & 1.0524E-05
& 1.9977        \\
160          & 3.5457E-06   & 1.9841       & 2.6306E-06
& 2.0002        \\
320          & 8.9028E-07   & 1.9937        &6.5767E-07
 & 2.0000        \\

   \hline
\end{tabular}
\end{table}\vspace{1mm}

\begin{table}[htbp]
\centering\small
\caption{
The errors $E_{l^2(J;V)}$ and the convergence orders of the constant step-size BDF2 (CSBDF2) method and the variable step-size BDF2 (VSBDF2) method with $\varpi=3$ for problem (\ref{eq6.1a}). Upper: The backward Euler method for the starting value $U^1$; Bottom: The trapezoidal formula for the starting value $U^1$.}
\begin{tabular}{|c|cc|cc|}\hline
 &  CSBDF2&   & VSBDF2 &($\varpi=3$)\\
\cline{2-3}\cline{4-5}
$N$     & Error      &Order  & Error&Order   \\\hline
20           & 2.0462E-03   &              & 7.6569E-04   &               \\
40           & 6.6766E-04   & 1.6158       & 1.9041E-04   & 2.0077        \\
80           & 2.0084E-04   & 1.7331       & 4.7440E-05   & 2.0049       \\
160          & 5.6422E-05   & 1.8317       & 1.1838E-05   & 2.0027       \\
320          & 1.5087E-05   & 1.9030       & 2.9566E-06   & 2.0014        \\
   \hline
20           & 5.8899E-04   &              & 7.6569E-04   &              \\
40           & 1.5458E-04   & 1.9299       & 1.9041E-04   & 2.0077       \\
80           & 3.9569E-05   & 1.9659       & 4.7440E-05   & 2.0049       \\
160          & 1.0003E-05   & 1.9839       & 1.1838E-05   & 2.0027       \\
320          & 2.5139E-06   & 1.9924       & 2.9566E-06   & 2.0014       \\

   \hline
\end{tabular}
\end{table}\vspace{1mm}
For CSBDF2 method, however, from these tables we observe that the errors of CSBDF2 with the starting BE step (\ref{eq2.9}) are larger than those of CSBDF2 with the starting TF step (\ref{eq2.4}) in all these norms. The order reduction phenomena are also observed in discrete $l^\infty(J;V)$ (Table 6.1) and $l^2(J;H,H)$ (Table 6.2) norms, especially, in discrete $l^2(J;H,H)$ norm, the order is less than $1.5$. These theoretical and numerical results presented in this paper suggest that for CSBDF2 method the backward Euler (BE) scheme (\ref{eq2.9}) is not the best choice.
\subsection{Experiment 2: nonlinear case}
In the second experiment we consider the $2$D semilinear equation,
\begin{eqnarray}\label{eq6.7}
u_t= \varepsilon(u_{xx}+u_{yy}) + u - u^3+g(t,x,y),\quad (x,y)\in [0,1]\times [0,1],\quad t\in (0,T],
\end{eqnarray}
with periodic boundary conditions. The function $g(t,x,y)$ and the initial value are selected in such a way that the exact solution is
$$u(t,x,y)=\sin(2\pi x)\cos(2\pi y)\exp(-\pi^2t).$$
It is easy to verify that the function $f(t,x,y,u)=u-u^3+g(t,x,y)$ satisfies condition (\ref{eq5.2a}) (see, e.g., \cite{Henry81,Czaja11}). We note that if $g\equiv 0$, equation (\ref{eq6.7}) is just the famous Allen-Cahn equation \cite{Allen79}, also known as the Chafee-Infante equation \cite{Chafee74}.

We consider a pseudo-spectral method for space discretization with $\Delta x=\Delta y=\frac{1}{M}$, where $M=256$. We let $T=1$ and $\varepsilon=0.01$. The numerical results are presented in Tables 6.5, 6.6, 6.7 and 6.8. For VSBDF2 method, the correct order of convergence is observed for all quantities. For CSBDF2 method, the order reduction phenomena in discrete $l^\infty(J;V)$ and $l^2(J;H,H)$ norms are still observed. We also notice that the CSBDF2 method with the starting TF step (\ref{eq2.4}) has much higher accuracy than the CSBDF2 method with the starting BE step (\ref{eq2.9}). To clearly illustrate this, the time evolutions of the discrete $L^2$ and $H^1$ errors are shown in Fig. 6.3.

\begin{table}[htbp]
\centering\small
\caption{
The errors $E_{l^\infty(J;V)}$ and the convergence orders of the constant step-size BDF2 (CSBDF2) method and the variable step-size BDF2 (VSBDF2) method with $\varpi=3$ for problem (\ref{eq6.7}). Upper: The backward Euler method for the starting value $U^1$; Bottom: The trapezoidal formula for the starting value $U^1$.}
\begin{tabular}{|c|cc|cc|}\hline
 &  CSBDF2&   & VSBDF2 &($\varpi=3$)\\
\cline{2-3}\cline{4-5}
$N$     & Error      &Order  & Error&Order   \\\hline
20           & 4.4379E-01   &              & 2.4029E-01   &               \\
40           & 1.3819E-01   & 1.6832        & 5.6710E-02   & 2.0831         \\
80           & 4.0324E-02   & 1.7770        & 1.3738E-02   & 2.0455        \\
160          & 1.1033E-02   & 1.8699        & 3.3798E-03   & 2.0231        \\
320          & 2.9251E-03   & 1.9152        & 8.3819E-04   & 2.0116         \\
   \hline
20           & 3.5795E-01   &              & 2.4035E-01   &              \\
40           & 9.7379E-02   & 1.8781        & 5.6711E-02   & 2.0834        \\
80           & 2.5390E-02   & 1.9393        & 1.3738E-02   & 2.0455        \\
160          & 6.4821E-03   & 1.9698        & 3.3798E-03   & 2.0231        \\
320          & 1.6375E-03   & 1.9849        & 8.3819E-04   & 2.0116        \\

   \hline
\end{tabular}
\end{table}\vspace{1mm}

\begin{table}[htbp]
\centering\small
\caption{
The errors $E_{l^2(J;H,H)}$ and the convergence orders of the constant step-size BDF2 (CSBDF2) method and the variable step-size BDF2 (VSBDF2) method with $\varpi=3$ for problem (\ref{eq6.7}). Upper: The backward Euler method for the starting value $U^1$; Bottom: The trapezoidal formula for the starting value $U^1$.}
\begin{tabular}{|c|cc|cc|}\hline
 &  CSBDF2&   & VSBDF2 &($\varpi=3$)\\
\cline{2-3}\cline{4-5}
$N$     & Error      &Order  & Error&Order   \\\hline
20           & 2.6970E-02   &              & 2.4749E-02   &               \\
40           & 1.4766E-02   & 0.8691       & 5.9608E-03   & 2.0538         \\
80           & 6.5878E-03   & 1.1644       & 1.4530E-03   & 2.0365        \\
160          & 2.6362E-03   & 1.3214       & 3.5813E-04   & 2.0205        \\
320          & 9.9438E-04   & 1.4066       & 8.8866E-05   & 2.0108         \\
   \hline
20           & 4.3060E-02   &              & 2.4750E-02   &              \\
40           & 1.2982E-02   & 1.7298       & 5.9608E-03   & 2.0538        \\
80           & 3.5827E-03   & 1.8574       & 1.4530E-03   & 2.0365        \\
160          & 9.4271E-04   & 1.9262       & 3.5813E-04   & 2.0205        \\
320          & 2.4192E-04   & 1.9623       & 8.8866E-05   & 2.0108        \\

   \hline
\end{tabular}
\end{table}\vspace{1mm}

\begin{table}[htbp]
\centering\small
\caption{
The errors $E_{l^\infty(J;H)}$ and the convergence orders of the constant step-size BDF2 (CSBDF2) method and the variable step-size BDF2 (VSBDF2) method with $\varpi=3$ for problem (\ref{eq6.7}). Upper: The backward Euler method for the starting value $U^1$; Bottom: The trapezoidal formula for the starting value $U^1$.}
\begin{tabular}{|c|cc|cc|}\hline
 &  CSBDF2&   & VSBDF2 &($\varpi=3$)\\
\cline{2-3}\cline{4-5}
$N$     & Error      &Order  & Error&Order   \\\hline
20           & 4.9904E-02   &              & 2.7042E-02
  &               \\
40           & 1.5543E-02   & 1.6829       & 6.3821E-03
  & 2.0831        \\
80           & 4.5355E-03   & 1.7769       & 1.5460E-03
 & 2.0455       \\
160          & 1.2413E-03   & 1.8694       & 3.8036E-04
 & 2.0231       \\
320          & 3.2914E-04   & 1.9151       & 9.4329E-05
  & 2.0116         \\
   \hline
20           & 4.0284E-02   &              & 2.7049E-02   &              \\
40           & 1.0959E-02   & 1.8781       & 6.3822E-03
& 2.0834         \\
80           & 2.8574E-03   & 1.9393       & 1.5460E-03
& 2.0455         \\
160          & 7.2949E-04   & 1.9698       & 3.8036E-04
& 2.0231         \\
320          & 1.8429E-04   & 1.9849        &9.4329E-05
 & 2.0116         \\

   \hline
\end{tabular}
\end{table}\vspace{1mm}

\begin{table}[htbp]
\centering\small
\caption{
The errors $E_{l^2(J;V)}$ and the convergence orders of the constant step-size BDF2 (CSBDF2) method and the variable step-size BDF2 (VSBDF2) method with $\varpi=3$ for problem (\ref{eq6.7}). Upper: The backward Euler method for the starting value $U^1$; Bottom: The trapezoidal formula for the starting value $U^1$.}
\begin{tabular}{|c|cc|cc|}\hline
 &  CSBDF2&   & VSBDF2 &($\varpi=3$)\\
\cline{2-3}\cline{4-5}
$N$     & Error      &Order  & Error&Order   \\\hline
20           & 3.8263E-01   &              & 2.4029E-01
      &               \\
40           & 1.0965E-01   & 1.8030        & 5.6710E-02   & 2.0831         \\
80           & 2.9232E-02   & 1.9073        & 1.3738E-02   & 2.0455        \\
160          & 7.5354E-03   & 1.9558        & 3.3798E-03   & 2.0231        \\
320          & 1.9119E-03   & 1.9786        & 8.3819E-04   & 2.0116         \\
   \hline
20           & 3.5795E-01   &              & 2.4035E-01   &              \\
40           & 9.7379E-02   & 1.8781        & 5.6711E-02   & 2.0834        \\
80           & 2.5390E-02   & 1.9393        & 1.3738E-02   & 2.0455        \\
160          & 6.4821E-03   & 1.9698        & 3.3798E-03   & 2.0231        \\
320          & 1.6375E-03   & 1.9849        & 8.3819E-04   & 2.0116        \\

   \hline
\end{tabular}
\end{table}\vspace{1mm}


\begin{figure}
\centering
\subfigure[Discrete $L^2$ errors, $N=320$, $M=256$]{ \label{fig:subfig:a}
\includegraphics[width=3in]{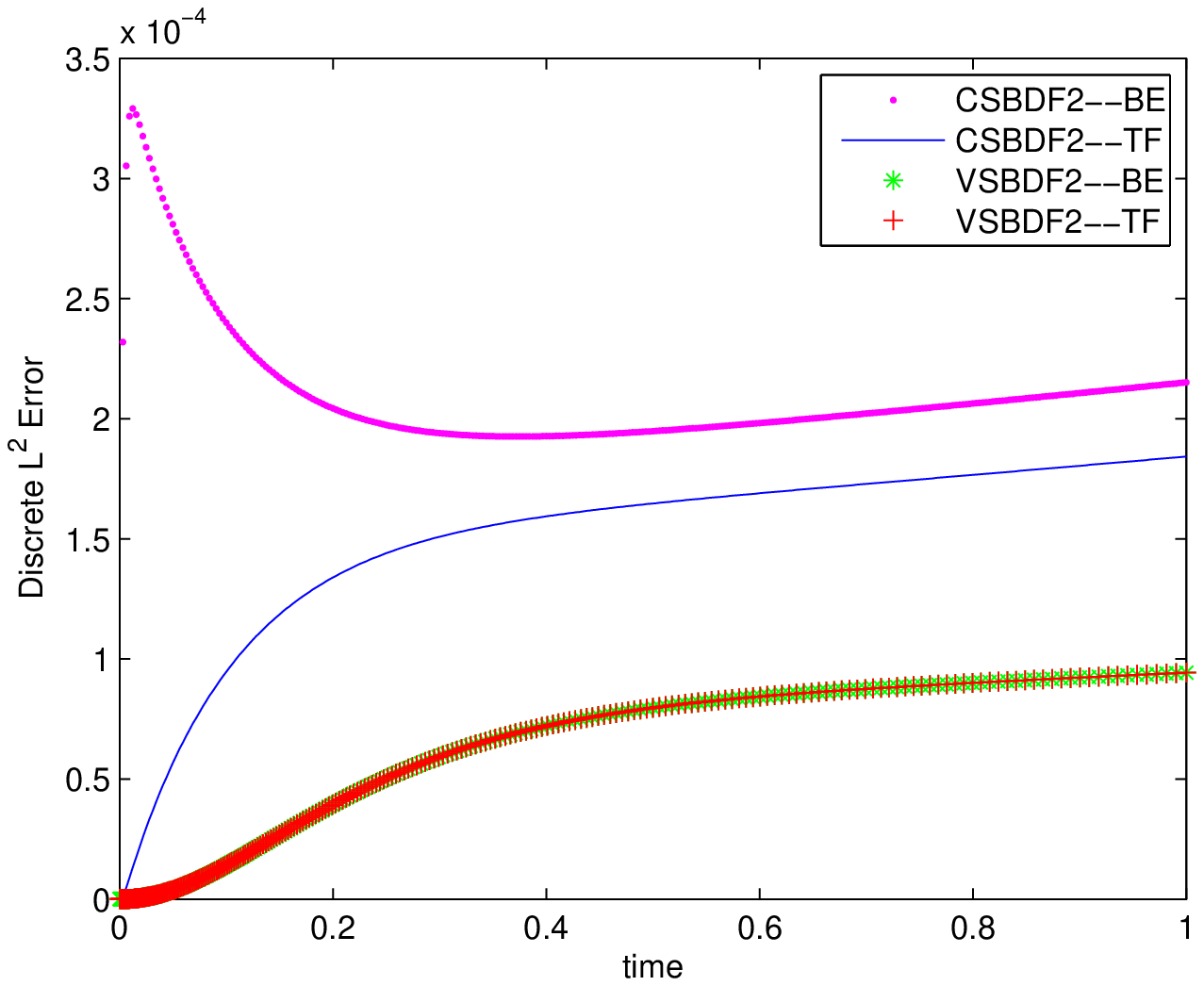}}
\hspace{0.01in}
\subfigure[Discrete $H^1$ errors, $N=320$, $M=256$]{ \label{fig:subfig:b}
 \includegraphics[width=3in]{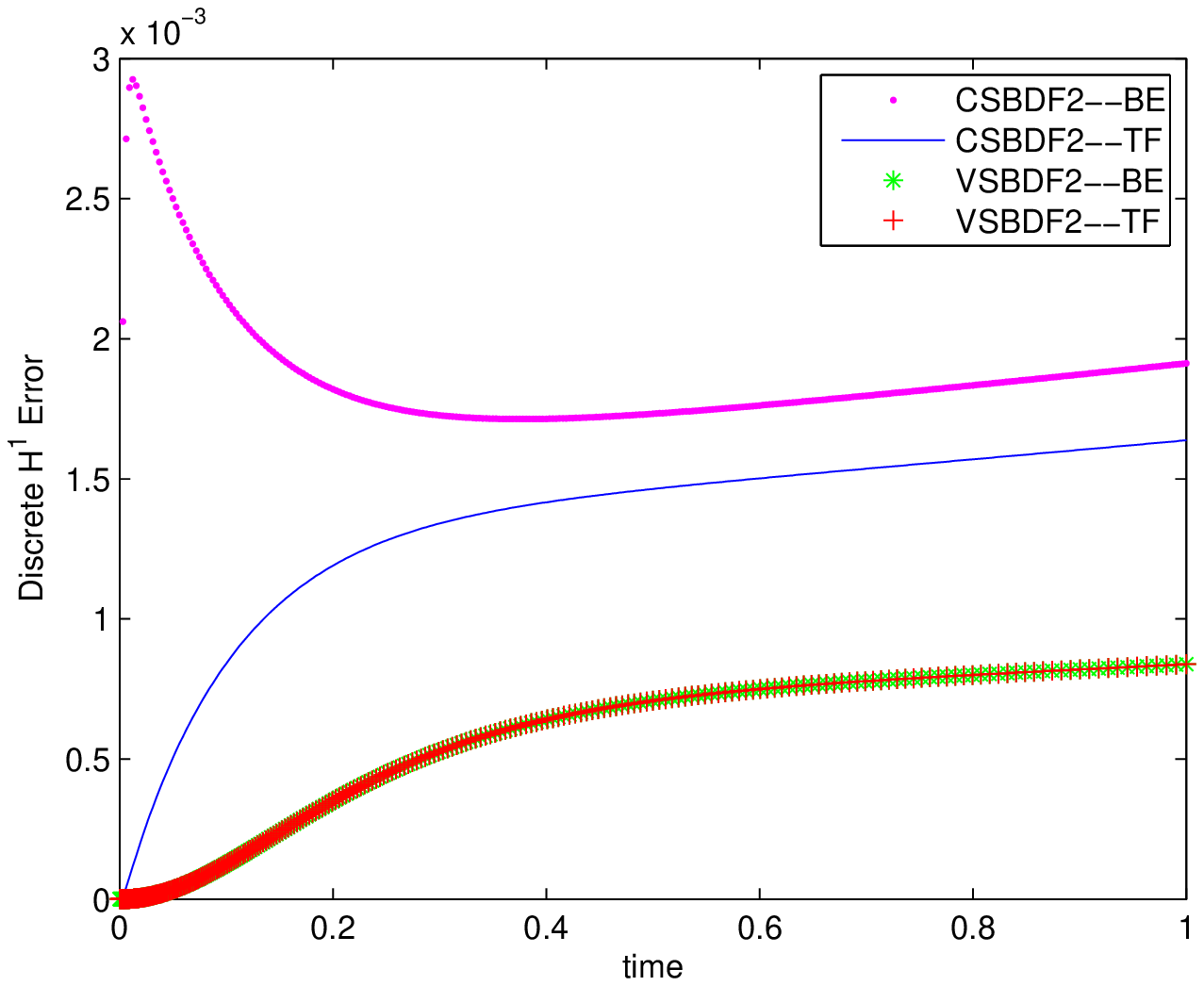}}
 \caption{Time evolution of the discrete errors of the constant step-size BDF2 (CSBDF2) method and the variable step-size BDF2 (VSBDF2) method ($\varpi=3$) for problem (\ref{eq6.7}), where $N=320$, $M=256$.} \label{fig:subfig}
 \end{figure}
\section{Concluding remarks}
In this work we considered the stability and error estimates of the variable step-size BDF2 method applied to linear and semilinear parabolic equations. We first obtained the $l^\infty(J;V)$ and $l^2(J;H,H)$-stabilities of the variable step-size BDF2 method for linear problems under the condition that the ratios of consecutive step-sizes are less than $R_0=\sqrt{2}+1\approx 2.414$, which is the zero stability condition. Then, using the $l^2(J;H,H)$-stability estimate, we proved that the upper bound of the step-size ratios for the $l^\infty(J;H)$-stability of the variable step-size BDF2 method for linear and semilinear parabolic equations is identical with the upper bound for the zero-stability for the first time. The bound of the step-size ratios can be also improved to $R_1=(3+\sqrt{17})/2\approx 3.561$ for the $l^\infty(J;V)$ and $l^2(J;V)$-stabilities. Based on the stability analysis and consistency error analysis, we derived global error bounds for the variable step-size BDF2 method in $l^\infty(J;V)$, $l^2(J;H,H)$, $l^\infty(J;H)$, and $l^2(J;V)$ norms. Since the variable step-size BDF2 method allows us take different time step-sizes for different time scales, i.e., small time step-sizes for the time domain with solution rapidly varying and large for the time domain with solution slowly changing, it demonstrates the prominent advantages of high accuracy compared to the constant step-size BDF2 method. To utilize the BDF method the trapezoidal method and the backward Euler scheme are employed to compute the starting value $U^1$. For the latter choice, order reduction phenomenon of the constant step-size BDF2 method is observed theoretically and numerically in the $l^\infty(J;V)$ and $l^2(J;H,H)$ norms. However, for the variable step-size BDF2 with the starting backward Euler step, the order reduction can be avoided by choosing a smaller starting step-size $k_1$.

We have implemented two numerical experiments for the variable step-size BDF2 method for linear and semilinear parabolic equations. For both equations these experiments exactly verify the theoretical results. The second order accuracy is maintained for time variable grid. These numerical experiments suggest that the variable step-size BDF2 method is more accurate than the popular constant step-size BDF2 method in several norms.


\end{document}